\documentclass[12pt,reqno,oneside]{article}

\usepackage{color}
\usepackage{graphics,epsfig}
\usepackage{graphicx}
\usepackage{float} 
\usepackage[english]{babel}
\usepackage{mathtools}
\usepackage{colonequals}
\usepackage{nicefrac}
\usepackage{enumitem}

\usepackage{dsfont, bm, amssymb, latexsym, graphicx, amsfonts, amsmath, booktabs, mathtools, amsthm, geometry, environ, multicol, subfiles, titlesec, titletoc, mathrsfs, fancyhdr, enumitem, xpatch, csquotes}

\newtheoremstyle{thmsty}{6pt}{6pt}{\normalfont\selectfont\itshape}{18pt}{\normalfont\selectfont\scshape}{.}{.5em}{}
\theoremstyle{thmsty}

\newtheorem{theorem}{Theorem}
\newtheorem{thm}{Theorem}
\newtheorem{lemma}[thm]{Lemma}

\newtheorem{cor}[thm]{Corollary}
\newtheorem{prop}[thm]{Proposition}

\numberwithin{theorem}{section}
\numberwithin{equation}{section}

\newcommand{\vecspan}{\operatorname{span}}

\numberwithin{equation}{section}
\numberwithin{theorem}{section}
\numberwithin{thm}{section}
\numberwithin{table}{section}

\def\vol {{\mathrm{vol\,}}}

\newcommand{\rank}{\operatorname{rank}}

\def \balpha{\bm{\alpha}}

\def \bgamma{\bm{\gamma}}

\def \blambda{\bm{\lambda}}

\def \bsigma{\bm{\sigma}}

\def \bnu{\bm{\nu}}
\def \ba{\mathbf{a}}
\def \bb{\mathbf{b}}
\def \bc{\mathbf{c}}
\def \bu{\mathbf{u}}
\def \bv{\mathbf{v}}
\def \bd{\mathbf{d}}
\def \be{\mathbf{e}}

\def \bs{\mathbf{s}}
\def \br{\mathbf{r}}

\def \bq{\mathbf{q}}

\def \by{\mathbf{y}}
\def\boldm{\mathbf{m}}

\def \bx{\mathbf{x}}
\def \bzero{\mathbf{0}}
\def \bone{\mathbf{1}}

\def \bfeta{\bm{\eta}}

\def \energy{\mathsf{E}}

\newcommand{\bfxi}{{\boldsymbol{\xi}}}

\def \dif{\mathrm{d}}

\def\cA{{\mathscr A}}
\def\cB{{\mathscr B}}
\def\cC{{\mathscr C}}
\def\cD{{\mathscr D}}
\def\cE{{\mathscr E}}

\def\cH{{\mathscr H}}
\def\cI{{\mathscr I}}
\def\cJ{{\mathscr J}}

\def\cL{{\mathscr L}}

\def\cN{{\mathscr N}}

\def\cR{{\mathscr R}}
\def\cS{{\mathscr S}}
\def\cT{{\mathscr T}}
\def\cU{{\mathscr U}}

\def\cX{{\mathscr X}}
\def\cY{{\mathscr Y}}

\def \A {{\mathbf{A}}}

\def \C {{\cc}}

\def \K {{\mathbf{K}}}

\def \N {{\nn}}
\def \P {{\mathbf P}}
\def \Q {{\qq}}
\def \R {{\rr}}
\def \Z {{\zz}}

\def\energy{\mathsf{E}}

\def\fB{\mathfrak{B}}

\def\fD{{\mathfrak D}}

\def\fI{{\mathfrak{Im}}\:}

\def\fP{{\mathfrak P}}
\def\fR{{\mathfrak{Re}}\:}
\def\fS{{\mathfrak  S}}

\def \fV{{\mathfrak V}}

\def\e{{\mathbf{\,e}}}

\def \matn{\N^{\boldm}}
\def \matz{\Z_+^{\boldm}}
\def \matc{\C^{\boldm}}
\def \matr{\R_+^{\boldm}}

\usepackage[style=numeric-comp, sorting=nyt]{biblatex}
\DeclareFieldFormat{labelnumberwidth}{#1\adddot}
\DeclareFieldFormat*{title}{\mkbibemph{#1}}
\DeclareFieldFormat*{journaltitle}{#1}
\DeclareFieldFormat*{bookTitle}{#1}
\DeclareFieldFormat[article,periodical]{volume}{\mkbibbold{#1}}
\DeclareFieldFormat*{pages}{#1}

\renewcommand*{\bibnamedash}{\rule[.4ex]{3em}{.6pt}}

\makeatletter
\newbool{bbx@inset}
\DeclareBibliographyDriver{set}{%
  \booltrue{bbx@inset}%
  \entryset{}{}%
  \newunit\newblock
  \usebibmacro{setpageref}%
  \finentry}

\renewbibmacro*{begrelated}{%
  \booltrue{bbx@inset}}

\renewbibmacro*{endrelated}{%
  \usebibmacro*{bbx:savehash}}

\InitializeBibliographyStyle{%
  \global\undef\bbx@lasthash}

\newbibmacro*{bbx:savehash}{%
  \savefield{fullhash}{\bbx@lasthash}}

\renewbibmacro*{author}{%
  \ifboolexpr{
    test \ifuseauthor
    and
    not test {\ifnameundef{author}}
  }
    {\usebibmacro{bbx:dashcheck}
       {\printtext{\bibnamedash}}
       {\printnames{author}%
        \setunit{\addcomma\space}%
        \usebibmacro{bbx:savehash}}%
     \usebibmacro{authorstrg}}
    {\global\undef\bbx@lasthash}}

\renewbibmacro*{editor}{%
  \usebibmacro{bbx:editor}{editorstrg}}
\renewbibmacro*{editor+others}{%
  \usebibmacro{bbx:editor}{editor+othersstrg}}
\newbibmacro*{bbx:editor}[1]{%
  \ifboolexpr{
    test \ifuseeditor
    and
    not test {\ifnameundef{editor}}
  }
    {\usebibmacro{bbx:dashcheck}
       {\printtext{\bibnamedash}}
       {\printnames{editor}%
        \setunit{\addcomma\space}%
        \usebibmacro{bbx:savehash}}%
     \usebibmacro{#1}%
     \clearname{editor}}
    {\global\undef\bbx@lasthash}}

\renewbibmacro*{translator}{%
  \usebibmacro{bbx:translator}{translatorstrg}}
\renewbibmacro*{translator+others}{%
  \usebibmacro{bbx:translator}{translator+othersstrg}}
\newbibmacro*{bbx:translator}[1]{%
  \ifboolexpr{
    test \ifusetranslator
    and
    not test {\ifnameundef{translator}}
  }
    {\usebibmacro{bbx:dashcheck}
       {\printtext{\bibnamedash}}
       {\printnames{translator}%
        \setunit{\addcomma\space}%
        \usebibmacro{bbx:savehash}}%
     \usebibmacro{#1}%
     \clearname{translator}}
    {\global\undef\bbx@lasthash}}

\newbibmacro*{bbx:dashcheck}[2]{%
  \ifboolexpr{
    test {\iffieldequals{fullhash}{\bbx@lasthash}}
    and
    not test \iffirstonpage
    and
    (
       not bool {bbx@inset}
       or
       test {\iffieldequalstr{entrysetcount}{1}}
    )
  }
    {#1}
    {#2}}
\makeatother

\renewbibmacro{in:}{%
  \ifentrytype{article}{}{\printtext{\bibstring{in}\intitlepunct}}}

\renewbibmacro*{journal+issuetitle}{%
  \usebibmacro{journal}%
  \setunit*{\addspace}%
  \iffieldundef{series}
    {}
    {\newunit
     \printfield{series}%
     \setunit{\addspace}}%
  \printfield{volume}%
  \setunit{\addspace}%
  \usebibmacro{issue+date}%
  \setunit{\addcolon\space}%
  \usebibmacro{issue}%
  \newunit
  \printfield{number}%
  \setunit{\addcomma\space}%
  \printfield{eid}
  \newunit}

\renewbibmacro{in:}{}

\newcommand\lle{\mathop{\:\ll\:}_\varepsilon}
\newcommand\nn{\boldsymbol{\mathrm{N}}}
\newcommand\zz{\boldsymbol{\mathrm{Z}}}
\newcommand\qq{\boldsymbol{\mathrm{Q}}}
\newcommand\rr{\boldsymbol{\mathrm{R}}}
\newcommand\cc{\boldsymbol{\mathrm{C}}}

\let\Gamma\varGamma
\let\Delta\varDelta
\let\Theta\varTheta
\let\Lambda\varLambda
\let\Xi\varXi
\let\Pi\varPi
\let\Sigma\varSigma
\let\Upsilon\varUpsilon
\let\Phi\varPhi
\let\Psi\varPsi
\let\Omega\varOmega

\xpatchcmd{\proof}{\hskip\labelsep}{\hskip4\labelsep}{}{}

\def\rint{\mathop{\rotatebox[origin=c]{15}{\scalebox{1.35}{$\int$}}}}

\titlespacing{\section}{18pt}{-5pt}{10pt}
\titleformat{\section}[runin]{\normalfont\bfseries}{\thesection.}{6pt}{}

\titlespacing{\subsection}{18pt}{-10pt}{10pt}
\titleformat{\subsection}[runin]{\normalfont\bfseries}{\thesubsection.}{4pt}{\normalfont\itshape}

\begin{document}
\centerline{\large{MOMENTS OF RESTRICTED DIVISOR FUNCTIONS}}\par\vspace{3.5mm}
\centerline{\small\scshape{M. Afifurrahman and C. C. Corrigan}}\par
\centerline{\textit{School of Mathematics and Statistics,}}\par
\centerline{\textit{University of New South Wales}}\par\vspace{3.5mm}
\centerline{\small{5$^{th}$\,of September, 2025}}\par\vspace{5mm}
{\small\noindent\textit{Abstract.} In this article, we study the higher-power moments of restricted divisor functions.  In order to establish our main results, we prove a more general result pertaining to the distribution of solutions to certain multiplicative Diophantine equations.\par\vspace{2mm}
\par\vspace{2mm}
}\par\vspace{5mm}

\section{Introduction.}
There are many important applications in analytic number theory for asymptotic formul\ae\:describing the distribution of solutions to Diophantine equations such as
\begin{align}\label{eqn:1111}
&x_1x_2=x_3x_4.
\end{align}
Indeed, the number of integer solutions to \eqref{eqn:1111} with $1\leqslant x_j\leqslant H$ is precisely the multiplicative energy of the first $H$ natural numbers, which appears in the theory of character sums (e.g.\:\cite{Ayyad}).  Moreover, the number of solutions to~\eqref{eqn:1111} over some finite subset $\mathscr{N}\subset\Z$ is precisely the number of singular $2\times 2$ integer matrices with entries in $\mathscr{N}$.  Therefore, results on the statistics of this equation complement results in the recent works~\cite{GG1,GG2,S,AKOS,Af} on the number of $2\times2$ integer matrices with bounded entries and fixed determinant.  Furthermore, the distribution of solutions to \eqref{eqn:1111} can be used to classify maximal-size product sets of random subsets of $\Z$ (e.g.\:\cite{M}), and, in short arithmetic progressions is also related to the problem of determining the irrational numbers $\alpha$ for which the pair correlation for the fractional parts of $n^2\alpha$ is Poissonian (e.g.\:\cite{Truelsen,S2}).

A more general result of Heath-Brown \cite{HB} can be used to show, for a certain class of smooth weights $w:\Z^4\to\R_+$, that an asymptotic formula of the type
\begin{equation*}
    \sum_{\substack{\bx\in\Z^4\\x_1x_2=x_3x_4}}w\big(X^{-1}\bx\big)=C_1(w)X^2\log X+C_2(w)X^2+O_\varepsilon(X^{3/2+\varepsilon})
\end{equation*}
holds for some constants $C_j(w)$, depending at most on $w$, where $C_1(w)>0$ can be explicitly calculated.  Comparatively, counting solutions to \eqref{eqn:1111} with respect to the supremum norm, it was shown by Ayyad, Cochrane, and Zheng \cite{Ayyad} (cf.\:\cite{M2,Af}) that\begin{align}\label{eqn:MA}
    \sum_{\substack{\bx\in[1,X]^4\\x_1x_2=x_3x_4}}1=\frac{2}{\zeta(2)}X^2\log X+AX^2+O\big(X^{19/13}(\log X)^{7/13}\big),
\end{align}where$$A=\frac{4\gamma-2\zeta'(2)/\zeta(2)-1-2\zeta(2)}{\zeta(2)},$$with $\gamma$ being the Euler–Mascheroni constant.  More recently, counting instead with respect to the Euclidean norm, Louboutin and Munsch \cite{louboutin} demonstrated that
\begin{equation}\label{e:loumun}
    \sum_{\substack{\bx\in\mathfrak{S}(X)\\x_1x_2=x_3x_4}}1=\tfrac34X^2\log X+O(X^2),
\end{equation}
where $\mathfrak{S}(X)$ denotes the first octant of the sphere of radius $X$ in $\zz^4$.  In fact, in private communications, Munsch has pointed out that, with little extra effort, one may replace the error term in \eqref{e:loumun} with a secondary term $cX^2+o_\eta(X^{2-\eta})$ for some explicit constant $c$.

We remark that the sum of interest in \eqref{eqn:MA} can be interpreted as the second moment of a certain restricted divisor function.  More generally, for any $\mathbf{X}\in\R_+^m$, the number of solutions to the Diophantine equation
\begin{equation*}
    x_{1,1}\cdots x_{1,m}=\cdots=x_{k,1}\cdots x_{k,m}
\end{equation*}
lying in the orthotope $\{\bx\in\nn^{k\times m}:x_{i,j}\leqslant X_j\}$ is precisely equal to the moment
\begin{equation*}
    M_{m,k}(\mathbf{X})\colonequals\sum_n\tau_m(n;\mathbf{X})^k,\quad\text{where}\quad\tau_m(n;\mathbf{X})\colonequals\mathop{\sum_{d_1\leqslant X_1}\cdots\sum_{d_m\leqslant X_m}}_{d_1\cdots d_m=n}1.
\end{equation*}
Note that $\tau_m(n;\mathbf{X})=0$ whenever $n>X_1\cdots X_m$, and thus the moment $M_{m,k}(\mathbf{X})$ is indeed well-defined.  Now, the number $M_{m,k}(\mathbf{X})$ was studied in the horizontal aspect by Munsch and Shparlinski~\cite{MSh}, who proved the existence of a constant $C_{m,\bc}>0$, depending at most on $m$ and $\bc\succ\bzero_m$, such that
\begin{align*}
    M_{m,2}(X^{c_1},\ldots,X^{c_m})\sim C_{m,\bc} X^{\|\bc\|}(\log X)^{(m-1)^2}\quad\text{as}\quad X\to\infty,
\end{align*}
where $\|\bc\|=|c_1|+\cdots+|c_m|$.  Conversely, in the vertical aspect, Mastrostefano~\cite{M2} showed that there exists a constant $D_k>0$ such that
\begin{equation*}
    M_{2,k}(X,X)\sim D_kX^2(\log X)^{2^k-k-1}\quad\text{as}\quad X\to\infty.
\end{equation*}
More generally, we may prove the following.
\begin{theorem}\label{thm:mainthm2}
    Fix integers $m,k\geqslant1$, and suppose that $\bc\in\Q_{>0}^m$.  Write
    \begin{equation*}
        q_{m,k,\bc}=V_{m,k,\bc}\prod_p(1-p^{-1})^{m^k}\sum_{n\geqslant0}\frac{1}{p^{n}}\binom{n+m-1}{m-1}^{\!\!k},
    \end{equation*}
    where, with $\cS_{m,k}=\{\bsigma\in\{0,1\}^{k\times m}:\|\bsigma_i\|=1\:\forall i\leqslant k\}$, we have written
    \begin{equation*}
        V_{m,k,\bc}=\vol\Big\{(v_{\bsigma})_{\bsigma\in\cS_{m,k}}\in\R_+^{\cS_{m,k}}:\sum_{\bsigma\in\cS_{m,k}}v_{\bsigma}\sigma_{i,j}=c_j\:\forall i\leqslant k\:\forall j\leqslant m\Big\}.
    \end{equation*}
    Then, there exists a polynomial $Q_{m,k,\bc}$, of degree precisely $m^k-(m-1)k-1$ and leading coefficient $q_{m,k,\bc}$, such that
    \begin{equation*}
        M_{m,k}(X^{c_1},\ldots,X^{c_m})=X^{\|\bc\|}Q_{m,k,\bc}(\log X)+o(X^{\|\bc\|-\vartheta_{m,k,\bc}})\quad\text{as}\quad X\to\infty,
    \end{equation*}
    for some $\vartheta_{m,k,\bc}>0$, where the implied constant depends at most on $m,k$, and $\bc$.
\end{theorem}
In the case $k=2$, the result above makes explicit the above mentioned result of Munsch and Shparlinski \cite{MSh}. In particular, if we additionally take $\bc=\bone_m$, then following an argument of Harper, Nikeghbali, and Radziwi\l\l\:\cite{HNR}, we see that
\begin{equation*}
    V_{m,2,\bone_m}=\binom{2(m-1)}{m-1}\frac{B_m}{m^{m-1}},
\end{equation*}
where $B_m$ is the $(m-1)^2$-dimensional volume of the Birkhoff polytope $\cB_m\subset\R_+^{m^2}$, for which an asymptotic formula has been established by Canfield and McKay \cite{birkhoff}.  In general, $V_{m,k,\bc}$ represents the $m^k-(m-1)k-1$ dimensional volume of a polytope embedded in an $m^k$ dimensional space.

We note that the method used to establish Theorem~\ref{thm:mainthm2} can also be applied to study the asymptotic behaviour of the more general expression
\begin{equation*}
    M_{m,k,\ell}(\mathbf{X})\colonequals\sum_n\tau_m(n^\ell;\mathbf{X})^k,
\end{equation*}
which is related to the distribution of solutions to the Diophantine equation
\begin{align*}
    x_{1,1}\dots x_{1,m}=\cdots=x_{k,1}\cdots x_{k,m}=y^{\ell}.
\end{align*}
To our knowledge, all previous works in this direction pertain to the case where $k=1$ or $\ell=1$.  In particular, the number $M_{2,1,\ell}(X,X)$ was studied by Tolev~\cite{T}, and the number $M_{3,1,3}(X,X)$, with additional coprimality conditions, by La~Bret\`{e}che~\cite{dlB98}, Fouvry \cite{fouvry}, and Heath-Brown and Moroz~\cite{HBM}.  More recently, it was shown by La~Bret\`{e}che, Kurlberg, and Shparlinski~\cite{dlBKS} that there exists a constant $\vartheta_{m,\ell}>0$, depending at most on $m$ and $\ell$, such that\begin{align}\label{eqn:dlBPSh}
    M_{m,1,\ell}(X,\ldots,X)= X^{m/\ell}Q_{m,\ell}(\log X) + O(X^{m/\ell-\vartheta_{m,\ell}})\quad\text{as}\quad X\to\infty,
\end{align}where $Q_{m,\ell}$ is a real polynomial with\begin{align*}
    \deg Q_{m,\ell}=\binom{\ell+m-1}{m-1}-m.
\end{align*} 
We further generalise their result in the following.
\begin{theorem}\label{thm:mainthm3}
    Fix integers $m,k,\ell\geqslant1$, and suppose that $\bc\in\Q_{>0}^m$.  Write
    \begin{equation*}
        q_{m,k,\ell,\bc}=V_{m,k,\ell,\bc}\prod_p(1-p^{-1})^{\binom{\ell+m-1}{m-1}^k}\sum_{n\geqslant0}\frac{1}{p^n}\binom{\ell n+m-1}{m-1}^{\!\!k}
    \end{equation*}
    where, with $\cS_{m,k,\ell}=\{\bsigma\in\{0,1,\ldots,\ell\}^{k\times m}:\|\bsigma_i\|=\ell\:\forall i\leqslant k\}$, we have written
    \begin{equation*}
        V_{m,k,\ell,\bc}=\vol\Big\{(v_{\bsigma})_{\bsigma\in\cS_{m,k,\ell}}\in\R_+^{\cS_{m,k,\ell}}:\sum_{\bsigma\in\cS_{m,k,\ell}}v_{\bsigma}\sigma_{i,j}= c_j\:\forall i\leqslant k\:\forall j\leqslant m\Big\}.
    \end{equation*}
    Then, there exists a polynomial $Q_{m,k,\ell,\bc}$, of degree precisely $\binom{\ell+m-1}{m-1}^{\!\!k}-(m-1)k-1$ and leading coefficient $q_{m,k,\ell,\bc}$, such that
    \begin{align*}
        M_{m,k,\ell}(X^{c_1},\ldots,X^{c_m})=X^{\|\bc\|/\ell}Q_{m,k,\ell,\bc}(\log X)+o(X^{1-\vartheta_{m,k,\ell,\bc}})\quad\text{as}\quad X\to\infty,
    \end{align*}
    for some $\vartheta_{m,k,\ell,\bc}>0$, where the implied constant depends at most on $m,k,\ell$, and $\bc$.
\end{theorem}
Taking $k=1$ and $\bc=\bone_m$ in the above result, we recover the estimate \eqref{eqn:dlBPSh}, though our expression for the area $V_{m,1,\ell,\bone_m}$ is written in a slightly different manner to that given in the article \cite{dlBKS}.  Note that $V_{m,k,\ell,\bc}$ is the $\binom{\ell+m-1}{m-1}^k-(m-1)k-1$ dimensional volume of a polytope embedded in a $\binom{\ell+m-1}{m-1}^k$ dimensional space.\newline

\section{Organisation and strategy.}
As a prerequisite for the demonstrations of our main assertions, Theorems~\ref{thm:mainthm2}~and~\ref{thm:mainthm3}, we shall prove a more general result, pertaining to the distribution of solutions to more general systems of Diophantine equations.  This result is stated in \S\,\ref{sec:thms}, with Theorem~\ref{thm:mainresult} being the most general result of this work, and  Theorem~\ref{thm:mostgen} being, under some additional conditions, the strongest result.  For a summary of the notations required for the proof, we refer the reader to \S\,\ref{sec:notes}.  In \S\,\ref{sec:prelim}, we outline the main tools for our demonstrations, which are presented in \S\,\ref{sec:demo}.  In particular, \S\,\ref{sec:labreteche} is dedicated to a discussion of two results of La~Bret\`eche \cite{dlB,dlBcompter}, which relate the statistics of an arithmetic function $f$ to the study of the multiple Dirichlet series associated to $f$.  We then, in \S\,\ref{sec:mds}, recall some tools for studying certain types of multiple Dirichlet series, which are essential in our application of the work of La~Bret\`eche.  Lastly, in \S\,\ref{sec:rem} we provide further remarks on our results, and suggest some problems as topics of potential future works.\newline

\subsection{Notations and conventions.}\label{sec:notes}
Firstly, we let $\mathbf{P}$, $\nn$, $\zz$, $\rr$, and $\cc$ denote the set of primes, natural numbers, integers, reals, and complex numbers, respectively.  Moreover, we adopt the convention that $\Z_+$ and $\R_+$ contain $0$, while $\N$ does not.  The letter $p$, with or without a subscript, will always denote a prime.  Similarly, any sum or product indexed by $p$ is intended to run over the primes.  Any other sum or product, unless otherwise obvious, is assumed to run over the natural numbers.  We shall use $\varepsilon$ to denote an arbitrarily small positive constant, which may vary in actual value from line to line.  We will additionally make use of both the usual Landau and Vinogradov notations $O,o,\sim,\ll,\gg$ (see e.g.\:\cite{davenport,montgomery}), as well as the following standard conventions:
\begin{itemize}[topsep=0pt]\setlength\itemsep{-0.3em}
    \item $\#\cS$ denotes the cardinality of a finite set $\cS$.    
    \item $\vecspan\cS$ denotes the span of the set $\cS$ over $\R$.
    \item $\vol\cT$ denotes the Lebesgue volume of $\cT$ in an appropriate embedded space.
\end{itemize}
In this article, we follow the convention that the volume of a region of dimension zero, i.e.\:a point, is 1.  Moreover, the volume of a line segment is the length of that line segment.  We also introduce the following notations, some of which are non-standard:
\begin{itemize}[topsep=0pt]\setlength\itemsep{-0.3em}
    \item $\cS^{\boldm}=\cS^{m_1}\oplus\cdots\oplus\cS^{m_k}$ denotes the set of $\br=\br_1\oplus\cdots\oplus\br_k$ with $\br_i\in\cS^{m_i}$.
    \item $r_{i,j}$ denotes the $j$-th ordinate of $\br_i$.
    \item $\br_1$, $\dots$, $\br_k$ are parts of $\br$.
    \item $\cI^{\boldm}$ denotes the set of integer pairs $(i,j)$ where $1\leqslant i\leqslant k$ and $1\leqslant j\leqslant m_i$.
    \item $\cH_{\ba}^{\boldm}$ denotes the subset of $\bs\in\C^{\boldm}$ for which $\fR s_{i,j}>a_{i,j}$ for all $i,j\in\cI^{\boldm}$.
    \item $\fR\bs$, where $\bs\in\C^{\boldm}$ is the $\bsigma\in\R^{\boldm}$ with $\fR s_{i,j}=\sigma_{i,j}$ for all $(i,j)\in\cI^{\boldm}$.
    \item $\fI\bs$ is defined similarly to $\fR\bs$.
    \item $\bone_{\boldm}=\bone_{m_1}\oplus\cdots\oplus\bone_{m_k}$, where $\bone_n$ denotes the $\by\in\C^n$ with $y_j=1$ for all $j\leqslant n$.
    \item $\bzero_{\boldm}=\bzero_{m_1}\oplus\cdots\oplus\bzero_{m_k}$, where $\bzero_n$ denotes the zero-tuple in $\C^n$.
    \item $\by^{\bfxi}=y_1^{\xi_1}\cdots y_n^{\xi_n}$ for any $\by,\bfxi\in\R^n$.
    \item $Y^{\bfxi}=(Y^{\xi_1},\ldots,Y^{\xi_n})$ for any $Y\in\R$ and $\bb\in\R^n$.
    \item $\cY^{\cX}$ denotes the set of sequences $(y_{\bx})_{\bx\in\cX}$ with $y_{\bx}\in\cY$ for all $\bx\in\cX$.
    \item $\langle\ba,\bb\rangle$ denotes the standard inner product on $\C^{\boldm}$, that is, the scalar product.
    \item $\|\ba\|$ denotes the Manhattan-norm of $\ba\in\R^{\boldm}$.
    \item $\cL^{\boldm}(\C)$ denotes the space of linear forms $\ell:\matc\to\C$.
    \item $\cL^{\boldm}_+(\C)$ denotes the subset of $\ell\in\cL^{\boldm}(\C)$ for which $\ell|_{\rr_+}$ maps to $\rr_+$.
    \item $\cE=\{\be_{i,j}\}$ denotes the standard basis for $\C^{\boldm}$ over $\C$.
    \item $\cE^*=\{\be_{i,j}^*\}$ denotes the standard dual basis in $\cL^{\boldm}(\C)$.
\end{itemize}
We will additionally write $\ba\succ\bb$ for some $\ba,\bb\in\R^{\boldm}$ if $a_{i,j}>b_{i,j}$ for all $(i,j)\in\cI^{\boldm}$, and we similarly define the relations $\ba\succeq\bb$, $\ba\prec\bb$, and $\ba\preceq\bb$.\newline

\subsection{A more general setting.}
In the following, we fix an integer $k\geqslant2$, suppose that $\boldm=(m_1,\dots,m_k)\in\N^k$, and for each $i\leqslant k$, suppose that $\bgamma_i\in\N^{m_i}$.  For convenience, we write $\bgamma=\bgamma_1\oplus\cdots\oplus\bgamma_k$, and denote the set of all such $\bgamma$ by $\matn$, and say that $\bgamma$ is an $\boldm$-tuple.  Moreover, we let $\gamma_{i,j}$ denote the $j^{th}$ ordinate of $\bgamma_i$, and for notational convenience, we let $\cI^{\boldm}$ denote the set of all indices $(i,j)$ such that $1\leqslant i\leqslant k$ and $1\leqslant j\leqslant m_i$.  As our method allows for it, we shall study a more general analog of multiplicative energy, given by
\begin{align*}
\energy_{\bgamma,\mathscr{N}}(H^{\bb})\colonequals\#\big\{\bx\in\mathfrak{B}^{\boldm}(H^{\bb})\cap\mathscr{N}:\bx_1^{\bgamma_1}=\cdots=\bx_k^{\bgamma_k}\big\},
\end{align*}
where $\mathscr{N}\subset\N^{\boldm}$ is closed under the Hadamard product, and
\begin{equation*}
    \mathfrak{B}^{\boldm}(H^{\bb})\colonequals\{\bx\in\nn^{\boldm}:x_{i,j}\leqslant H^{b_{i,j}}\:\forall(i,j)\in\mathscr{I}^{\boldm}\}
\end{equation*}
for some fixed $\bb\in\rr_+^{\boldm}$.  Now, the reason why we consider the solutions in the box $\fB^{\boldm}(H^{\bb})$, rather than a more arbitrary bounded subset of $\N^{\boldm}$, is due to the fact that the method we use to study $\energy_{\bgamma,\mathscr{N}}(H^{\bb})$ produces an asymptotic formula in the most interesting, balanced case where\begin{align}\label{eqn:N}
    \langle\bgamma_1,\bb_1\rangle=\cdots=\langle\bgamma_k,\bb_k\rangle=N
\end{align} for some rational number $N$.  Here, and in the remainder of the article, we let $\langle\cdot,\cdot\rangle:\fV\times\fV\to\R$ denote the scalar product on $\fV=\R^m$ or $\fV=\R^{\boldm}$, whichever is appropriate.  Before moving on, we briefly note that the set of solutions to
\begin{equation}\label{eqn:x1xk}
    \bx_1^{\bgamma_1}=\cdots=\bx_k^{\bgamma_k}
\end{equation}
can be interpreted as the collection of integral points on the affine variety defined by \eqref{eqn:x1xk}. This is in some way related to Manin's conjecture (cf.\:\cite{FMT}), which concerns the number of rational points on a suitably nice Zariski-open subset $\mathscr{U}$ of a projective variety $\mathscr{X}$ defined over a number field $\K$.  In particular, Manin's conjecture predicts that, for an appropriate height function $\mathfrak{h}:\mathscr{X}\to\rr_+$, we should have$$\#\{\bx\in \mathscr{U}(\K) \colon \mathfrak{h}(\bx) \leqslant H\} \sim cH^{A}(\log H)^{B}\quad\text{as}\quad H\to\infty,$$ for some constants $A$, $B$, and $c$, depending on the geometry of $\mathscr{X}$.  This conjecture has been confirmed with respect to several broad families of varieties, and we refer the interested reader to the articles \cite{FMT,BT,Santens,dlBB1,dlBB2,brow,brownheathbrown,blomer} for more information in this direction.  Note that, in this context, $\energy_{\gamma,\N^{\boldm}}(H^{\bone_{\boldm}})$ is related to the number of integral points, of bounded height, on the singular toric variety\begin{align*}
       \big\{\bx \in \A^{\boldm} \colon \bx_{i_1}^{\bgamma_{i_1}}-\bx_{i_2}^{\bgamma_{i_2}}=0,\:\forall i_1<i_2\leqslant k\big\},
    \end{align*}with respect to the height function \begin{align*}
        \mathfrak{h}:\bx\mapsto\max_{(i,j)\in \cI^{\boldm}}|x_{i,j}|.
    \end{align*}
    While our results in \S\,\ref{sec:thms} will pertain only to the case where $\mathscr{N}\subseteq\N^{\boldm}$, a simple combinatorial argument allows us to extend them to subsets of $\Z^{\boldm}$, and thus be related to the above (cf.\:\S\,\ref{sec:singularmat}).\newline

\subsection{Trivial bounds.}  In this section, for the sake of convenience, we only consider the case where $\mathscr{N}=\N^{\boldm}$, wherein we write $\energy_{\bgamma,\N^{\boldm}}(H^{\bb})=\energy_{\bgamma}(H^{\bb})$, though it is clear that similar bounds hold in general.  Firstly, we note that
\begin{align}\label{eqn:upperbound}
    \energy_{\bgamma}(H^{\bb})&\leqslant\min_{t\leqslant k}\sum_{\bx_t\in\mathfrak{B}^{m_t}(H^{\bb_t})}\prod_{\substack{i=1\\i\neq t}}^k\#\big\{\bx_i\in\N^{m_i}:\bx_i^{\bgamma_i}=\bx_t^{\bgamma_t}\big\}\notag\\&\leqslant\min_{t\leqslant k}\sum_{\bx_t\in\mathfrak{B}^{m_t}(H^{\bb_t})}\max_{i\neq t}\#\big\{\bd\in\N^{\|\bgamma_i\|}:d_1\cdots d_{\|\bgamma_i\|}=\bx_t^{\bgamma_t}\big\}^{k-1}\notag\\&\lle\min_{t\leqslant k} H^{\|\bb_t\|+\varepsilon},
\end{align}
by virtue of the standard bound for the divisor function (e.g.\:(1.81) in \cite{IK}), and the fact that $\#\fB^{m_t}(H^{\bb_t})\sim H^{\|\bb_t\|}$ as $H\to\infty$.  Note that, in the important case where $\bb_i=\bone_{m_i}$, we have $\|\bb_i\|=m_i$.

A trivial lower-bound can easily be derived for the case where $\bgamma$ satisfies~\eqref{eqn:N} and $\alpha\bb\in\N^{\boldm}$ for some $\alpha\in\N$, that is to say when $\bb$ has rational ordinates. These conditions naturally appear in many number theoretical applications. Trivially,
\begin{equation*}
     \energy_{\bgamma}(H^{\bb})\geqslant H^{1/\alpha}+O(1),
\end{equation*}
which follows from the fact that $\bx=h^{\alpha\bb}$ is a solution to \eqref{eqn:x1xk} for all $h\leqslant H^{1/\alpha}$.  To improve this bound, we must firstly develop some nomenclature.  Let $\mu_{\bgamma,\bb}$ be the greatest number $\mu$ such that there exist tuples $\bfeta^1,\ldots,\bfeta^{\mu}\in\{0,1\}^{\boldm}\backslash\{\bzero_{\boldm}\}$, satisfying $\bfeta^1+\cdots+\bfeta^{\mu}=\bone_{\boldm}$, such that
\begin{equation}\label{eqn:mugamma99}
    \langle\bgamma_1,\bfeta^u_1*\bb_1\rangle=\cdots=\langle\bgamma_k,\bfeta_k^u*\bb_k\rangle=N_u\quad\text{for all}\quad u\leqslant\mu,
\end{equation} 
and some fixed $N_1,\ldots,N_\mu\in\Q$, where $*$ denotes the Hadamard product.  Note that we necessarily have $N_1+\cdots+N_\mu=N$, where $N$ is as in~\eqref{eqn:N}.  Moreover, in the important case where $\bb=\bone_{m_1}\oplus\cdots\oplus\bone_{m_k}$, clearly $\mu_{\bgamma,\bone_{\boldm}}$ is precisely the maximal possible partitions of the coordinates of $\bgamma$ to collections of subsets with equal sums (cf.\:\S\,1.2 of \cite{AfCo}).  With these notations, it is clear that$$\bx=d_1^{\bfeta^1*\alpha\bb}*\cdots*d_{\mu_{\bgamma,\bb}}^{\bfeta^{\mu_{\bgamma,\bb}}*\alpha\bb}\in\N^{\boldm}$$ is a solution to \eqref{eqn:x1xk} for any natural numbers $d_1,\ldots,d_{\mu_{\bgamma,\bb}}\leqslant H^{1/\alpha}$.  Therefore, we obtain the stronger lower bound
\begin{align}\label{eqn:lowerbound}
    \energy_{\bgamma}(H^{\bb})&\geqslant\#\big\{d_1^{\bfeta^1*\alpha\bb}*\cdots*d_{\mu_{\bgamma,\bb}}^{\bfeta^{\mu_{\bgamma,\bb}}*\alpha\bb}:d_1,\ldots,d_{\mu_{\bgamma,\bb}}\leqslant H^{1/\alpha}\big\}\notag\\
    &=H^{\mu_{\bgamma,\bb}/\alpha}+O\big(H^{(\mu_{\bgamma,\bb}-1)/\alpha}\big).
\end{align}
In fact, as we will see in \S\,\ref{sec:nonexplicit}, the bound \eqref{eqn:lowerbound} is of the correct order of magnitude for certain choices of $\bgamma$ and $\bb$.  Moreover, comparing~\eqref{eqn:upperbound} and~\eqref{eqn:lowerbound}, it is clear that $1\leqslant\mu_{\bgamma,\bb}\leqslant\alpha\|\bb_i\|$ for all $i \leqslant k$. 
With respect to these bounds, our main objective in the remainder of this paper is to make improvements for particular cases of $\bgamma$.\newline 

\subsection{The main contention.}\label{sec:thms}
Our results, as with those existing in the literature, relate the study of the number $\energy_{\gamma,\mathscr{N}}(H^{\bb})$ to the study of the set
\begin{equation*}
    \cR_{\bgamma}\colonequals\{\br\in\Z_+^{\boldm}:\langle\bgamma_1,\br_1\rangle=\cdots=\langle\bgamma_k,\br_k\rangle\}
\end{equation*}
of solutions to the associated additive Diophantine equation.  Our first result provides a uniform treatment of all multiplicative Diophantine equations of the form \eqref{eqn:x1xk}.
\begin{theorem}\label{thm:mainresult}
    Fix $\boldm\in\N^k$, and suppose that $\bgamma\in\N^{\boldm}$ and $\bb\in\R_{>0}^{\boldm}$.  With $\cR_{\bgamma}$ defined as above, suppose that $\mathscr{N}\subseteq\N^{\boldm}$ is closed under the Hadamard product, and such that for all $\br\in\mathscr{R}_{\bgamma}$ the product
    \begin{equation}\label{eqn:notfriable}
        \prod_p\frac{p^\sigma-1}{p^\sigma-\psi_{\mathscr{N}}(p^{\br})}
    \end{equation}
    is absolutely convergent and non-zero for some $\sigma\in(0,1)$, where $\psi_{\mathscr{N}}:\N^{\boldm}\to\{0,1\}$ denotes the characteristic function of $\mathscr{N}$.  Suppose moreover that $\ba\succ\bzero_{\boldm}$ has real ordinates and is such that $\langle\ba,\br\rangle\geqslant1$ for all non-zero $\br\in\cR_{\bgamma}$, and write
    \begin{equation*}
        \kappa_{\bgamma}(\ba)\colonequals\#\cR_{\bgamma}(\ba)-\rank\cR_{\bgamma}(\ba)\quad\text{where}\quad\cR_{\bgamma}(\ba)\coloneqq\{\br\in\cR_{\bgamma}:\langle\ba,\br\rangle=1\}.
    \end{equation*}
    Then, there exists a polynomial $Q$, of degree not exceeding $\kappa_{\bgamma}(\ba)$, such that
    \begin{equation}\label{eqn:igorcomment}
        \energy_{\bgamma,\mathscr{N}}(H^{\bb})=H^{\langle\ba,\bb\rangle}Q(\log H)+o(H^{\langle\ba,\bb\rangle-\vartheta})\quad\text{as}\quad H\to\infty
    \end{equation}
    for some $\vartheta>0$, depending at most on $\ba,\bb$, and $\bgamma$.
\end{theorem}
The above statement is derived from a result of La~Bret\`eche~\cite{dlB,dlBcompter} pertaining to sums of arithmetic functions of many variables.  Note that the polynomial $Q$ is not necessarily non-zero, and thus \eqref{eqn:igorcomment} is not an asymptotic formula for every choice of $\ba$, as is to be expected.  Indeed, the reason why the right-hand side of \eqref{eqn:igorcomment} depends on $\ba$, while the left-hand side does not, is due to the fact that $\bgamma$ explicitly determines all of the $\ba$ for which \eqref{eqn:igorcomment} yields an asymptotic formula.  Indeed, supposing that $\ba$ and $\bb$ satisfy some easily verifiable criteria, we can not only determine exactly the degree of $Q$, but also the leading coefficient.  In this direction, we have the following.
\begin{theorem}\label{thm:mostgen}
    Assume the hypothesis of Theorem~\ref{thm:mainresult}, and moreover suppose that $\ba\in\Q_{>0}^{\boldm}$ is chosen such that
    \begin{equation}\label{eqn:gammaisa}
        \sum_{\br\in\cR_{\bgamma}(\ba)}\br\Z_+=\cR_{\bgamma},
    \end{equation}
    and that $\bb\in\Q_{>0}^{\boldm}$ is such that
    \begin{equation}\label{eqn:gammaisb}
        \langle\bgamma_1,\bb_1\rangle=\cdots=\langle\bgamma_k,\bb_k\rangle.
    \end{equation}
    For any natural number $n$ and prime $p$, define the number 
    \begin{equation*}
        \nu_{\bgamma,\cN}(\ba;n,p)\colonequals\#\{\br\in\cR_{\bgamma,\cN}(p):\langle\ba,\br\rangle=n\}
    \end{equation*}
    where
    \begin{equation*}
        \cR_{\bgamma,\cN}(p)\colonequals\big\{\br\in\cR_{\bgamma}:p^{\br}\in\cN\:\forall(i,j)\in\cI^{\boldm}\big\}.
    \end{equation*}
    Then, the polynomial $Q$ is of degree precisely $\kappa_{\bgamma}(\ba)$, and moreover satisfies
    \begin{equation*}
        Q(\log H)\sim(\log H)^{\kappa_{\bgamma}(\ba)}V_{\bgamma}(\ba,\bb)\prod_p(1-p^{-1})^{\#\cR_{\bgamma}(\ba)}\sum_{n\geqslant0}\frac{\nu_{\bgamma,\mathscr{N}}(\ba;n,p)}{p^n}\quad\text{as}\quad H\to\infty,
    \end{equation*}
    where
    \begin{align*}
        V_{\bgamma}(\ba,\bb)\colonequals\vol\Big\{(v_{\br})_{\br}\in\R_+^{\cR_{\bgamma}(\ba)}:\sum_{\br\in\cR_{\bgamma}(\ba)}r_{i,j}v_{\br}=b_{i,j}\:\forall(i,j)\in\cI^{\boldm}\Big\}.
    \end{align*}
\end{theorem}
Note that the condition \eqref{eqn:gammaisa} essentially requires that the ordinates of $\ba$ are all sufficiently small that $\|\ba\|\leqslant\|\ba'\|$ for any $\ba'\succ\bzero_{\boldm}$ satisfying $\langle\ba',\br\rangle\geqslant1$ for all non-zero $\br\in\cR_{\bgamma}$.  Moreover, to avoid technical complications, we consider the most interesting, balanced case \eqref{eqn:gammaisb}, although the method works in greater generality.  We additionally remark that $V_{\bgamma}(\ba,\bb)$ is to be interpreted as the $\kappa_{\bgamma}(\ba)$-dimensional volume of a surface contained in an embedding of $\R^{\kappa_{\bgamma}(\ba)}$ in $\R^{\cR_{\bgamma}(\ba)}$.\newline

\section{Preliminary Lemmata.}\label{sec:prelim}
As in \cite{dlBKS,MSh,M}, our main tools will come from the theory of multiple Dirichlet series.  In order to reformulate our problem in a way that allows us to make use of this theory, we appeal to the work of La~Bret\`eche \cite{dlB,dlB98}.  Essentially, his work pertains to the asymptotic behaviour of sums of the form     
    \begin{align*}
        S_f(X^{\bb})=\sum_{\bone\preceq\bd\preceq X^{\bb}} f(\bd),
    \end{align*}
    where $\bb\succ\bzero_{\boldm}$ and $f:\matn\to\R_+$ are fixed.  Thus, taking $f$ to be the characteristic function of some set $\cD\subset\matn$, we see that $S_f(X^{\bb}$) simply counts the number of $\bd\in\cD$ with $d_{i,j}\leqslant X^{b_{i,j}}$ for all $(i,j)\in\cI^{\boldm}$.\newline

\subsection{Two results of La~Bret\`eche.}\label{sec:labreteche}
      In order to properly state the aforementioned work of La~Bret\`eche, we must firstly develop some nomenclature.  With $f:\N^{\boldm}\to\R_+$ as above, we define $\cA_f$ to be the collection of $\ba\in\matr$ such that, for any sufficiently small $\delta,\delta'>0$, the following three properties are satisfied:\vspace{-2mm}
        \begin{enumerate}[label={\upshape(\roman*)},itemsep=0mm,leftmargin=10mm]
            \item The multiple Dirichlet series
            \begin{equation*}
                D_f(\bs)\colonequals\sum_{\bd\in\matn}f(\bd)\bd^{-\bs}
            \end{equation*}   
            converges absolutely in the domain $\cH_{\ba}^{\boldm}$ ;
            \item There exists a finite family $\cL_f(\ba)\subset\cL_+^{\boldm}(\C)$ of non-zero linear forms, such that the function $G_f(\ba,\cdot):\cH_{\bzero_{\boldm}}\to \C$, defined by the map
            \begin{align*}
                G_f(\ba,\cdot):\bs\mapsto D_f(\bs+\ba)\prod_{\ell\in\cL_f(\ba)} \ell(\bs),
            \end{align*}
            can be analytically continued to the domain $\cD_{\cL_f(\ba)}(0,\delta,\infty)$, where
            \begin{equation*}
                \cD_{\cL_f(\ba)}(\sigma,\delta,\Delta)\colonequals\bigcap_{\ell\in\cL_f(\ba)}\{\bs\in\cH_{-\delta\ba}^{\boldm}:\sigma-\delta<\langle\fR\bs,\br\rangle<\sigma+\Delta\}.
            \end{equation*}
            \item For all $\varepsilon,\sigma,\Delta>0$, there exists a constant $c_f(\ba;\varepsilon,\sigma,\delta,\delta',\Delta)>0$ such that
            \begin{align*}
                |G_f(\ba,\bs)|\leqslant c_f(\ba;\varepsilon,\sigma,\delta,\delta',\Delta)(1+\|\fI\bs\|)^\varepsilon\prod_{\ell\in\cL_f(\ba)} (1+|\fI\ell(\bs)|)^{1-\delta'\min(0,\fR\ell(\bs))}
            \end{align*}
            holds uniformly for $\bs$ in the domain $\cD_{\cL_f(\ba)}(\sigma,\delta,\Delta)$.            
        \end{enumerate}
        We note that the result originally stated by La~Bret\`eche in Th\'eor\`eme~1 of \cite{dlB} is actually stronger than the following, and is moreover stated in a slightly different manner.  Alternate versions of this result are also stated in \cite{M2,dlBKS,MSh,dlBcompter}, though the form which is sufficient for our purposes is as follows.
\begin{lemma}\label{lem:dlB}
    With the above notations, fix a real $\bb\succ\bzero_{\boldm}$ and an $f:\matn\to\R_+$ which is such that $\bzero_{\boldm}\notin\mathscr{A}_f$.  Take $\ba\in\cA_f$ and put $\cC_f(\ba)=\cL_f(\ba)\cup\cE(\ba)$, where
    \begin{equation*}
        \cE(\ba)=\{\be_{i,j}^*:(i,j)\in\cJ(\ba)\}\quad\text{with}\quad\cJ(\ba)=\{(i,j)\in\cI^{\boldm}:a_{i,j}=0\}.
    \end{equation*}    
    Then, there exists a $\vartheta>0$ and a real polynomial $P$, of degree not exceeding \begin{align*}\nu_f(\ba)=\#\cC_f(\ba)-\rank\cC_f(\ba),\end{align*} such that
    \begin{align*}
        S_f(X^{\bb})=X^{\langle\ba,\bb\rangle}P(\log X)+o(X^{\langle\ba,\bb\rangle-\vartheta})
    \end{align*}
    uniformly as $X\to\infty$.
    \end{lemma}
    Note that the polynomial $P$ in the above is not necessarily not the zero-polynomial.  Nonetheless, we propose in \S\,\ref{sec:nonexplicit} a procedure for determining the non-triviality of $P$.  Moreover, under some reasonable additional conditions, the leading coefficient of $P$ can be explicitly determined using the following result, which is an immediate consequence of Th\'eor\`eme~2 of \cite{dlB} (cf.\:Th\'eor\`eme~B of \cite{dlBcompter}).
    \begin{lemma}\label{lem:dlB2}
    Assume the hypothesis of Lemma~\ref{lem:dlB}, and moreover suppose that $\ell(\ba)=1$ for all $\ell\in\cL_f(\ba)$, and that there exist real $\beta_\ell>0$ such that
    \begin{align}\label{eqn:betacond}
        \sum_{(i,j)\in\cI^{\boldm}}b_{i,j}\be_{i,j}^*=\sum_{\ell\in\cC_f(\ba)}\beta_\ell\ell.
    \end{align}
    With $G_f(\ba,\bs)$ as defined in condition $\mathrm{(ii)}$ above, suppose that there exists a function $G_f^*(\ba;\cdot):\C^{\cC_f(\ba)}\to\C$ such that 
    \begin{equation}\label{eqn:gstarcond}G_f(\ba,\bs)=G_f^*(\ba;(\ell(\bs))_{\ell\in\cC_f(\ba)}).\end{equation}  Then, we necessarily have
    \begin{align*}
        P(\log X)=\frac{G_f(\ba,\bzero_{\boldm})\:\vol\Omega_{\cL_f(\ba)}(X^{\bb})}{X^{\langle\ba,\bb\rangle}}+O\big((\log X)^{\nu_f(\ba)-1}\big)\quad\text{as}\quad X\to\infty,
    \end{align*}
    where
    \begin{equation*}
        \Omega_{\cL_f(\ba)}(X^{\bb})\colonequals\Big\{\by\in(1+\R_+)^{\cL_f(\ba)}:\prod_{\ell\in\cL_f(\ba)} y_\ell^{\ell(\be_{i,j})}\leqslant X^{b_{i,j}}\:\forall(i,j)\in\cI^{\boldm}\Big\}.
    \end{equation*}
\end{lemma}
It should be noted that the above result is a consequence of part of the statement of Th\'eor\`eme~2 of \cite{dlB}, and is in fact incorrectly applied in the works of Munsch and Shparlinski \cite{MSh} and Mastrostefano \cite{M2}, wherein a result similar to Lemma~\ref{lem:dlB2} above should have been utilised.  Nonetheless, the results in both of the aforementioned works can be recovered by appealing to Theorem~\ref{thm:mainthm2} of this article.\newline

\subsection{Preliminaries on multiple Dirichlet series.}\label{sec:mds}
Supposing that $\fV\cong\N^{n_1}$ and $\fS\cong\C^{n_2}$ for some $n_1,n_2\in\N$, the theory of multiple Dirichlet series is primarily concerned with the analytic properties of series of the form
\begin{equation*}
    \sum_{\mathfrak{v}\in\fV}a_{\mathfrak{v}}\exp(-\varrho_{\mathfrak{v}}(\bs)),
\end{equation*}
where $a_{\mathfrak{v}}\in\C$ and $\varrho_{\mathfrak{v}}:\fS\to\C$ is linear for all $\mathfrak{v}\in\fV$.  A complete treatment of this topic is not available in the literature, though the works \cite{kohji,sahoo,Bum,BFH,CFH,DFH} provide a number of fundamental results pertaining to certain types of these series.  In this article, we are primarily concerned with the type of multiple Dirichlet series studied by La~Bret\`{e}che~\cite{dlB,dlBcompter,dlB98} and Onozuka \cite{onozuka}.  In particular, we are interested in studying series of the type $D_f(\bs)$ in the cases where $f:\matn\to\R_+$ is polynomially bounded, that is, where there exists a $\bfxi\in\matn$ and a $c>0$ for which $0\leqslant f(\bd)\leqslant c\bd^{\bfxi}$ for all $\bd\in\matn$.  We are interested specifically in the case where $f$ is multiplicative, that is, in the case where, for any two distinct primes $p$ and $q$, we have
\begin{align*}
f(p^{\mathbf{u}}*q^{\mathbf{v}})=f(p^{\mathbf{u}})f(q^{\mathbf{v}})
\end{align*}
for all $\textbf{u},\textbf{v}\in\Z_+^{\boldm}$.  Now, in the case where $f$ is multiplicative, we may express $D_f(\bs)$ as an Euler product of the form
\begin{align}\label{eqn:F2}
    D_f(\bs)= \prod_{p}D_f(\bs;p)\quad\text{where}\quad D_f(\bs;p)=\sum_{\mathbf{r}\in\matz}f(p^{\br})p^{-\langle{\bs,\br\rangle}}.
\end{align}
It is well-known (e.g.\:\cite{balazard,delange,tenenbaum}) that a multiple Dirichlet series and its associated Euler product converge and diverge simultaneously, and it is through the Euler product representation that the absolute convergence of $D_f(\bs)$ can be determined by studying an appropriate additive function.  This gives rise to the relation between the solutions of the multiplicative Diophantine equation \eqref{eqn:x1xk} and the solutions $\mathscr{R}_{\bgamma}$ to the associated additive Diophantine equation.

Through the Euler product representation, the study of the series $D_f$ can be related in some form to the study of certain $L$-functions.  Indeed, the standard method of developing ratios conjectures for families of $L$-functions makes critical use of the theory of multiple Dirichlet series, e.g.\:in the recent work of \v{C}ech \cite{cech1,cech2} and Gao and Zhao \cite{gaozhao1,gaozhao2}.  For applications in the analytic theory of $L$-functions, however, it is not always necessary to know precisely what the region of absolute convergence for $D_f$ is, only that it exists and satisfies certain properties.  These properties can often be deduced from a classical result of Bochner \cite{bochner,bochner2}.  

Of the following two results, the first describes explicitly what the region of absolute convergence of $D_f$ is, and the second shows that the properties (ii) and (iii) of Lemma~\ref{lem:dlB} are satisfied for our specific choice of $f$.
\begin{lemma}\label{lem:absconv}
    Suppose that $\bgamma$, $\cN$, and $\psi_{\mathscr{N}}$ are as in Theorem~\ref{thm:mainresult}, and define the multiplicative function $f_{\bgamma,\mathscr{N}}:\N^{\boldm}\to\{0,1\}$ by the relation $f_{\bgamma,\cN}(p^{\br})=\chi_{\bgamma}(\br)\psi_{\cN}(p^{\br})$ for all $\br\in\Z_+^{\boldm}$ and primes $p$, where $\chi_{\bgamma}$ is the characteristic function of $\mathscr{R}_{\bgamma}$.  Then, the multiple Dirichlet series $D_{f_{\bgamma,\cN}}$ is absolutely convergent precisely in the region
    \begin{equation*}
        \cT_{\bgamma,\cN}=\bigcap_{\br\in\cR_{\bgamma}^*}\{\bs\in\matn:\langle\fR\bs,\br\rangle>1\},
    \end{equation*}
    where $\cR_{\bgamma}^*$ is the set of non-zero elements of $\cR_{\bgamma}$.
\end{lemma}
\begin{lemma}\label{lem:analytic}
    Under the hypothesis of Lemma~\ref{lem:absconv}, suppose that $\ba\in\matr$ is on the boundary of $\cT_{\bgamma,\cN}$, and let $\cR_{\bgamma}(\ba)$ denote the set of $\br\in\cR_{\bgamma}$ for which $\langle\ba,\br\rangle=1$.
    Then, the function $g_{\bgamma,\cN}(\ba,\cdot):\cH_{\bzero_{\boldm}}^{\boldm}\to\C$, defined by the map
    \begin{equation*}
        g_{\bgamma,\cN}(\ba,\cdot):\bs\mapsto D_{f_{\bgamma,\cN}}(\bs+\ba)\prod_{\br\in\cR_{\bgamma}(\ba)}\langle\bs,\br\rangle,
    \end{equation*}
    can be analytically continued to a holomorphic function in the domain $\cH_{-\delta\ba}^{\boldm}$, provided that $\delta>0$ is sufficiently small.  
    
    Moreover, $g_{\bgamma,\cN}(\ba,\bzero_{\boldm})$ is non-zero, and for any $\varepsilon,\sigma,\Delta>0$, there exists a constant $c_{\bgamma}(\ba;\varepsilon,\sigma,\delta,\Delta)>0$ such that the bound
    \begin{equation*}
        |g_{\bgamma,\cN}(\ba,\bs)|\leqslant c_{\bgamma}(\ba;\varepsilon,\sigma,\delta,\Delta)\prod_{\br\in\cR_{\bgamma}(\ba)}(1+|\langle\fI\bs,\br\rangle|)^{1-\frac13\min(0,\langle\fR\bs,\br\rangle)+\varepsilon}
    \end{equation*}
    holds uniformly in the domain $\cD_{L_{\bgamma}(\ba)}(\sigma,\delta,\Delta)$, where $$L_{\bgamma}(\ba)=\{\ell:\bs\mapsto\langle\bs,\br\rangle:\br\in\cR_{\bgamma}(\ba)\}.$$
\end{lemma}
The demonstrations of the above two results constitute the main effort required to establish Theorems~\ref{thm:mainresult}~and~\ref{thm:mostgen}.  Indeed, Lemmata~\ref{lem:absconv}~and~\ref{lem:analytic} can be used to find an optimal choice for $\ba$ in Lemmata~\ref{lem:dlB}~and~\ref{lem:dlB2}.  Now, the above lemmata can be established following a fairly standard argument, the general principle of which is given on p.\,265 of \cite{dlB}.  For the sake of completeness, though, we include the details here.  We firstly prove the following two results pertaining to the region of absolute convergence of $D_f$ and the positivity of $G_f$ on the boundary of this region.
\begin{lemma}\label{lem:rep}
    Assume the hypothesis of Lemma~\ref{lem:absconv} to hold, and suppose that $\ba\in\matr$ lies in the closure of the region of absolute convergence of $D_f$.  Write
    \begin{equation*}
        \cB_{\bgamma}(\ba)=\{\br\in\cR_{\bgamma}:\langle\ba,\br\rangle=\vartheta(\ba)\}\quad\text{where}\quad\vartheta(\ba)=\inf_{\br\in\cR^*_{\bgamma}}\langle\ba,\br\rangle.
    \end{equation*}
    Then $\vartheta(\ba)\geqslant1$, and there exists a holomorphic function $\phi_f(\ba,\cdot):\cH_{(1-\delta)\ba}^{\boldm}\to\cc$ and a real number $\delta>0$ such that
    \begin{equation*}
        D_f(\bs)=\phi_f(\ba,\bs)\prod_{\br\in\cB_{\bgamma}(\ba)}\zeta(\langle\bs,\br\rangle)\prod_{p}\frac{1-p^{-\langle\bs,\br\rangle}}{1-\psi_{\cN}(p^{\br})p^{-\langle\bs,\br\rangle}}\quad\text{for all}\quad\bs\in\cH_{\ba}^{\boldm}.
    \end{equation*}
\end{lemma}
\begin{proof}
    Firstly, note that there exists an $\eta(\ba)>\vartheta(\ba)$ such that $\langle\fR\ba,\br\rangle\geqslant\eta(\ba)$ for all $\br\in\cR_{\bgamma}$ with $\langle\fR\ba,\br\rangle\neq\vartheta(\ba)$.  Without loss of generality, suppose that $\langle\fR\ba,\br\rangle=\eta(\ba)$ for at least one $\br\in\cR_{\bgamma}$, in which case we see for any $\bs\in\cH_{\ba}^{\boldm}$ that
    \begin{equation}\label{eqn:dfsplit}
        D_f(\bs;p)=1+\sum_{\br\in\cB_{\bgamma}(\ba)}\psi_{\cN}(p^{\br})p^{-\langle\bs,\br\rangle}+O\big(p^{-\eta(\ba)}\big).
    \end{equation}
    Now, appealing to the fact that
        \begin{equation*}
            \prod_{\br\in\cB_{\bgamma}(\ba)}\big(1-\psi_{\cN}(p^{\br})p^{-\langle\bs,\br\rangle}\big)=1-\sum_{\br\in\cB_{\bgamma}(\ba)}\psi_{\cN}(p^{\br})p^{-\langle\bs,\br\rangle}+O\big(p^{-\eta(\ba)}\big),
        \end{equation*}
        we see from \eqref{eqn:dfsplit} that
        \begin{equation}\label{eqn:approxrep}
            D_f(\bs;p)\prod_{\br\in\cB_{\bgamma}(\ba)}\big(1-\psi_{\cN}(p^{\br})p^{-\langle\bs,\br\rangle}\big)=1+R_f(\ba,\bs;p),
        \end{equation}
        where $R_f(\ba,\bs;p)\ll p^{-\eta(\ba)}$ for any $\bs\in\cH_{\ba}^{\boldm}$, with the implied constant depending on $f$ and $\ba$ alone.  Now, we note that the product
        \begin{equation}\label{eqn:conv}
            \prod_{p\leqslant X}\big(1+R_f(\ba,\bs;p)\big)
        \end{equation}
        is absolutely convergent to some complex number $\phi_{f}(\ba,\bs)$ as $X\to\infty$, provided that $\eta(\ba)>1$.  Consequently, taking $\delta>0$ to be sufficiently small that $\vartheta(\ba)<(1-\delta)\eta(\ba)$, we see that \eqref{eqn:conv} is absolutely convergent in the region $\cH_{(1-\delta)\ba}^{\boldm}$, and thus therein defines a holomorphic function $\phi_{f}(\ba,\bs)$.  Moreover, we note that the product
        \begin{equation*}
            \prod_{p\leqslant X}\prod_{\br\in\cB_{\bgamma}(\ba)}\frac{1}{1-\psi_{\cN}(p^{\br})p^{-\langle\bs,\br\rangle}}
        \end{equation*}
        converges absolutely to
        \begin{equation}
            \prod_{\br\in\cB_{\bgamma}(\ba)}\zeta(\langle\bs,\br\rangle)\prod_{p}\frac{1-p^{-\langle\bs,\br\rangle}}{1-\psi_{\cN}(p^{\br})p^{-\langle\bs,\br\rangle}}
        \end{equation}
        as $X\to\infty$, provided that $\bs\in\cH_{\ba}^{\boldm}$ and $\vartheta(\ba)\geqslant1$.  This also conversely implies that we must have $\vartheta(\ba)\geqslant1$.  Thus, the assertion follows from \eqref{eqn:F2} and \eqref{eqn:approxrep}, provided that $\delta$ is sufficiently small that $\langle\bs,\br\rangle>\sigma$ for all $\bs\in\cH_{(1-\delta)\ba}^{\boldm}$ and all $\br\in\cB_{\bgamma}(\ba)$, where $\sigma\in(0,1)$ is as in Lemma~\ref{lem:absconv}.  This completes the demonstration.
\end{proof}
\begin{lemma}\label{lem:zerofree}
    Under the hypothesis of Lemma~\ref{lem:rep}, $\phi_f(\ba,\ba)$ is non-zero.
\end{lemma}
\begin{proof}
    Firstly, if $\br$ can be expressed as a linear combination of elements of $\cB_{\bgamma}(\ba)$ over $\Z_+$, then it surely must be an element of $\cR_{\bgamma}$.  Conversely, for any $\br\in\cR_{\bgamma}$, we have the trivial estimate
    \begin{equation*}
        \#\Big\{(w_{\bx})_{\bx\in\cB_{\bgamma}(\ba)}\in\{0,1\}^{\cB_{\bgamma}(\ba)}:\sum_{\bx\in\cB_{\bgamma}(\ba)}w_{\bx}\bx=\br\Big\}\leqslant2^{\#\cB_{\bgamma}(\ba)},
    \end{equation*}
    and thus, for any sequence $(p_{\br})_{\br\in\cB_{\bgamma}(\ba)}\in\P^{\cB_{\bgamma}(\ba)}$, we must have
    \begin{equation*}
        \#\Big\{(q_{\br})_{\br\in\cB_{\bgamma}(\ba)}\in\P^{\cB_{\bgamma}(\ba)}:\prod_{\br\in\cB_{\bgamma}(\ba)}q_{\br}^{\br}=\prod_{\br\in\cB_{\bgamma}(\ba)}p_{\br}^{\br}\Big\}\leqslant 2^{\#\cB_{\bgamma}(\ba)^2}.
    \end{equation*}
    Consequently, restricting the sum over $\bd\in\N^{\boldm}$ in the definition of $D_f(\bs)$ to certain prime-power tuples, we have by positivity the bound
    \begin{align}\label{eqn:dsgeq}
        D_f(\fR\bs)&\geqslant\frac{1}{2^{\#\cB_{\bgamma}(\ba)^2}}\prod_{\br\in\cB_{\bgamma}(\ba)}\sum_p\psi_{\cN}(p^{\br})p^{-\langle\fR\bs,\br\rangle}\notag\\&=\frac{1}{2^{\#\cB_{\bgamma}(\ba)^2}}\prod_{\br\in\cB_{\bgamma}(\ba)}\zeta(\langle\fR\bs,\br\rangle)\xi_{\br}(\fR\bs)
    \end{align}
    for some functions $\xi_{\br}$, holomorphic and non-zero on the closure of $\cH_{\ba}^{\boldm}$, with the property that $\xi_{\br}(\fR\bs)>0$ for all $\bs\in\cH_{\ba}^{\boldm}$.  Applying Lemma~\ref{lem:rep} to \eqref{eqn:dsgeq}, we get
    \begin{align*}
        \phi_f(\ba,\fR\bs)&=D_f(\fR\bs)\prod_{\br\in\cB_{\bgamma}(\ba)}\frac{1}{\zeta(\langle\fR\bs,\br\rangle)}\prod_p\frac{1-\psi_{\cN}(p^{\br})p^{-\langle\fR\bs,\br\rangle}}{1-p^{-\langle\fR\bs,\br\rangle}}\\
        &\geqslant\frac{1}{2^{\#\cB_{\bgamma}(\ba)^2}}\prod_{\br\in\cB_{\bgamma}(\ba)}\xi_{\br}(\fR\bs)\prod_p\frac{1-\psi_{\cN}(p^{\br})p^{-\langle\fR\bs,\br\rangle}}{1-p^{-\langle\fR\bs,\br\rangle}}
    \end{align*}
    for any $\bs\in\cH_{\ba}^{\boldm}$, by virtue of the fact that $\zeta(s)$ is well-defined and non-zero whenever $\fR s>1$.  Thus, as the map $\bs\mapsto\phi_f(\ba,\bs)$ is holomorphic at $\bs=\ba$, we must have
    \begin{equation*}
        \phi_f(\ba,\ba)=\lim_{\bs\to\ba}\phi_f(\ba,\fR\bs)\geqslant \frac{1}{2^{\#\cB_{\bgamma}(\ba)^2}}\prod_{\br\in\cB_{\bgamma}(\ba)}\big(\liminf_{\bs\to\ba}\xi_{\br}(\fR\bs)\big)\prod_p\frac{p-\psi_{\cN}(p^{\br})}{p-1}>0,
    \end{equation*}
    by virtue of the fact that each $\xi_{\br}$ is non-zero on the closure of $\cH_{\ba}^{\boldm}$.  This completes the proof.
\end{proof}
\begin{proof}[Proof of Lemma~\ref{lem:absconv}.]
    Firstly note that, by virtue of Lemmata~\ref{lem:rep}~and~\ref{lem:zerofree}, $\bs\in\matc$ is in the closure of the region of absolute convergence of $D_f$ if and only if $\vartheta(\fR\bs)\geqslant1$.  Consequently, the closure of the aforementioned region is precisely
    \begin{equation*}
        \bigcap_{\br\in\cR_{\bgamma}^*}\{\bs\in\matc:\langle\fR\bs,\br\rangle\geqslant1\}.
    \end{equation*}
    The assertion now follows from Lemma~\ref{lem:zerofree} and the fact that the $\zeta$-function has a simple pole at the point $s=1$.
\end{proof}
\begin{proof}[Proof of Lemma~\ref{lem:analytic}.]
    Appealing to Lemma~\ref{lem:rep} and the fact that the $\zeta$-function has a simple pole at $s=1$ and is holomorphic elsewhere, we see that
    \begin{equation}\label{eqn:dstarmap}
        G_f(\ba,\cdot):\bs\mapsto\phi_f(\ba,\ba+\bs)\prod_{\br\in\cB_{\bgamma}(\ba)}\zeta(1+\langle\bs,\br\rangle)\langle\bs,\br\rangle\prod_{p}\frac{1-p^{-1-\langle\bs,\br\rangle}}{1-\psi_{\cN}(p^{\br})p^{-1-\langle\bs,\br\rangle}}
    \end{equation}
    can be analytically continued to a holomorphic function in the domain $\cH_{-\delta\ba}^{\boldm}$, where $\delta>0$ as in Lemma~\ref{lem:rep}.  This proves the first part of the assertion.  The second part of the assertion, that $G_f(\ba,\bzero_{\boldm})$ is non-zero, follows from the fact that $G_f(\ba,\bs)$ is holomorphic at $\bs=\bzero_{\boldm}$ and therefore \eqref{eqn:dstarmap} gives
    \begin{align*}
        G_f(\ba,\bzero_{\boldm})&=\lim_{\bs\to\bzero_{\boldm}}\phi_f(\ba,\ba+\bs)\prod_{\br\in\cB_{\bgamma}(\ba)}\zeta(1+\langle\bs,\br\rangle)\langle\bs,\br\rangle\prod_{p}\frac{1-p^{-1-\langle\bs,\br\rangle}}{1-\psi_{\cN}(p^{\br})p^{-1-\langle\bs,\br\rangle}}\\
        &=\phi_f(\ba,\ba)\prod_{\br\in\cB_{\bgamma}(\ba)}\prod_{p}\frac{1-p^{-1}}{1-\psi_{\cN}(p^{\br})p^{-1}},
    \end{align*}
    which is non-zero by virtue of Lemma~\ref{lem:zerofree}.\par
    Now, assume without loss of generality that $2\delta<1$.  Then, appealing to the Vinogradov-Korobov estimate for the $\zeta$-function (e.g.\:Corollary 8.28 of \cite{IK}), we see that the bound
    \begin{equation*}
        \zeta(1+s)s\lle(1+|\fI s|)^{1-\frac13\min(0,\fR s)+\varepsilon}
    \end{equation*}
    holds uniformly in the vertical strip $\{s\in\C:\sigma-\delta<\fR s<\sigma+\Delta\}$, where the implied constant depends on at most $\varepsilon$, $\sigma$, $\delta$, and $\Delta$.  Consequently, appealing to \eqref{eqn:dstarmap}, we have the majorisation
    \begin{equation*}
        G_f(\ba,\bs)\lle|\phi_f(\ba,\ba+\bs)|\prod_{\br\in\cB_{\bgamma}(\ba)}(1+|\langle\fI\bs,\br\rangle|)^{1-\frac13\min(0,\langle\fR\bs,\br\rangle)+\varepsilon}
    \end{equation*}
    for any $\bs\in\fD_{\cL_f(\ba)}(\sigma,\delta,\Delta)$, where the implied constant depends on at most $\varepsilon$, $\sigma$, $\delta$, $\Delta$, and $\cB_{\bgamma}(\ba)$.  The proof is then complete on noting that $|\phi_f(\ba,\ba+\bs)|$ is uniformly bounded in the region $\cH_{-\delta\ba}^{\boldm}$ by a constant depending only on $f$, $\ba$, and $\delta$.  This fact, of course, follows from the definition of $\phi_f(\ba,\bs)$ in \eqref{eqn:conv}.
\end{proof}
From the above discussions, we see that the condition on the product \eqref{eqn:notfriable} in Theorem~\ref{thm:mainresult} could be replaced with a slightly weaker condition.  This, however would result in us having to assume an unreasonably strong hypothesis in place of \eqref{eqn:gammaisa} for the corresponding Theorem~\ref{thm:mostgen}.  We elaborate more on this point in \S\,3.2 of \cite{AfCo}.\newline

\section{Demonstrations.}\label{sec:demo}
We start this section by proving Theorems~\ref{thm:mainresult}~and~\ref{thm:mostgen}, from which we later derive Theorems~\ref{thm:mainthm2}~and~\ref{thm:mainthm3} as consequences.
\begin{proof}[Proof of Theorem~\ref{thm:mainresult}.]
    Taking $f_{\bgamma,\cN}$ as in Lemma~\ref{lem:absconv}, we clearly have
    \begin{align}\label{eqn:sumS}
        \energy_{\bgamma,\mathscr{N}}(H^{\bb})=\#\{\bx\in\mathfrak{B}^{\boldm}(H^{\bb}):f_{\bgamma,\cN}(\bx)=1\}=S_{f_{\bgamma,\cN}}(H^{\bb}).
    \end{align}
    By virtue of Lemmata~\ref{lem:absconv}~and~\ref{lem:analytic},  the hypothesis of Lemma~\ref{lem:dlB} is satisfied if we take $\cL_{f_{\bgamma,\cN}}(\ba)=L_{\bgamma}(\ba)$ and $G_{f_{\bgamma,\cN}}=g_{\bgamma,\cN}$, where $L_{\bgamma}(\ba)$ and $g_{\bgamma,\cN}$ are as in Lemma~\ref{lem:analytic}.  The assertion then follows immediately from Lemma~\ref{lem:dlB} and \eqref{eqn:sumS}, on noting that $\rank L_{\bgamma}(\ba)=\rank\cR_{\bgamma}(\ba)$.
\end{proof}
\begin{proof}[Proof of Theorem~\ref{thm:mostgen}.]
    We firstly note that the additional conditions imply that any $\br\in\cR_{\bgamma}$ can be constructed from a linear combination of $\br'\in\cR_{\bgamma}(\ba)$ over $\Z_+$, so that the condition \eqref{eqn:gstarcond} of Lemma~\ref{lem:dlB2} is satisfied.  Moreover, by virtue of \eqref{eqn:gammaisa}, \eqref{eqn:gammaisb}, and the fact that $\bb$ has rational ordinates, it is clear that for any $\br\in\cR_{\bgamma}(\ba)$ there exists a natural number $n_{\bb,\br}$ such that $n_{\bb,\br}\bb-\br\in\cR_{\bgamma}$, and thus there must exist integers $\lambda_{\br}(\br')\geqslant0$ such that
    \begin{equation*}
        n_{\bb,\br}\bb-\br=\sum_{\br'\in\cR_{\bgamma}(\ba)}\lambda_{\br}(\br')\br'.
    \end{equation*}
    Consequently, we have
    \begin{align*}
        \sum_{\br\in\cR_{\bgamma}(\ba)}n_{\bb,\br}\bb&=\sum_{\br\in\cR_{\bgamma}(\ba)}\br+\sum_{\br\in\cR_{\bgamma}(\ba)}\sum_{\br'\in\cR_{\bgamma}(\ba)}\lambda_{\br}(\br')\br'\\
        &=\sum_{\br\in\cR_{\bgamma}(\ba)}\br\Big(1+\sum_{\br'\in\cR_{\bgamma}(\ba)}\lambda_{\br'}(\br)\Big),
    \end{align*}
    so, letting $\br_\ell$ denote the unique element of $\cR_{\bgamma}(\ba)$ associated to $\ell\in L_{\bgamma}(\ba)$, we have
    \begin{equation*}
        \sum_{(i,j)\in\cI^{\boldm}}b_{i,j}\be_{i,j}^*=\sum_{\ell\in L_{\bgamma}(\ba)}\beta_{\ell}\ell\quad\text{where}\quad\beta_\ell=\Big(\sum_{\br\in\cR_{\bgamma}(\ba)}n_{\bb,\br}\Big)^{-1}\Big(1+\sum_{\br\in\cR_{\bgamma}(\ba)}\lambda_{\br}(\br_\ell)\Big)>0.
    \end{equation*}
    Therefore, the condition \eqref{eqn:betacond} of Lemma~\ref{lem:dlB2} is satisfied.  Hence, appealing to the fact that $g_{\bgamma,\cN}(\ba,\bzero_{\boldm})$ is non-zero, by virtue of Lemma~\ref{lem:analytic}, we see that all the conditions of Lemma~\ref{lem:dlB2} are satisfied for $f_{\bgamma,\cN}$ by taking $\cL_{f_{\bgamma,\cN}}(\ba)=L_{\bgamma}(\ba)$ and $G_{f_{\bgamma,\cN}}=g_{\bgamma,\cN}$, where $L_{\bgamma}(\ba)$ and $g_{\bgamma,\cN}$ are as in Lemma~\ref{lem:analytic}.  Thus, the polynomial $Q$ must satisfy
    \begin{equation}\label{eqn:middlesteps}
        Q(\log H)=\frac{g_{\bgamma,\cN}(\ba,\bzero_{\boldm})\:\vol\Omega_{L_{\bgamma}(\ba)}(H^{\bb})}{H^{\langle\ba,\bb\rangle}}+O\big((\log H)^{\kappa_{\bgamma}(\ba)-1}\big)
    \end{equation}
    as $H\to\infty$.\par  
    Now, for any $\sigma>0$, we have
    \begin{align}\label{eqn:gchi}
        &g_{\bgamma,\cN}(\ba,\sigma\bone_{\boldm})\notag\\
        &\hspace{5mm}=D_{f_{\bgamma,\cN}}(\ba+\sigma\bone_{\boldm})\prod_{\br\in\cR_{\bgamma}(\ba)}\sigma\|\br\|\notag\\
        &\hspace{5mm}=\Big(\prod_{\br\in\cR_{\bgamma}(\ba)}\sigma\|\br\|\Big)\prod_p\sum_{n_1\geqslant0}\sum_{n_2\geqslant0}\frac{\nu_{\bgamma,\cN}(\ba;n_1,n_2,p)}{p^{\sigma n_1+n_2}}\notag\\
        &\hspace{5mm}=\Big(\prod_{\br\in\cR_{\bgamma}(\ba)}\sigma\|\br\|\zeta(1+\sigma\|\br\|)\Big)\prod_p\sum_{n_1\geqslant0}\sum_{n_2\geqslant0}\frac{(1-p^{-1-\sigma\|\br\|})\nu_{\bgamma,\cN}(\ba;n_1,n_2,p)}{p^{\sigma n_1+n_2}},
    \end{align}
    since, for any $\delta>0$, both $\zeta(1+\delta)$ and $1/\zeta(1+\delta)$ have absolutely convergent Euler products.  Here, $\nu_{\bgamma,\cN}(\ba;n_1,n_2,p)$ denotes the number of $\br\in\cR_{\bgamma,\cN}(p)$ with $\|\br\|=n_1$ and $\langle\ba,\br\rangle=n_2$.  Consequently, on noting that the map $\bs\mapsto g_{\bgamma,\cN}(\ba,\bs)$ is, by virtue of Lemma~\ref{lem:analytic}, holomorphic at $\bs=\bzero_{\boldm}$, and that $\zeta$ has a simple pole at $s=1$ of residue 1, we derive from \eqref{eqn:gchi} that    
    \begin{align}\label{eqn:gchigamma}
        g_{\bgamma,\cN}(\ba,\bzero_{\boldm})=\prod_p(1-p^{-1})^{\#\cR_{\bgamma}(\ba)}\sum_{n\geqslant0}\frac{\nu_{\bgamma,\cN}(\ba;n,p)}{p^{n}}.
    \end{align}\par
    It now remains only to establish an asymptotic expression for $\vol\Omega_{L_{\bgamma}(\ba)}(H^{\bb})$, for then Theorem~\ref{thm:mostgen} would follow from \eqref{eqn:middlesteps}~and~\eqref{eqn:gchigamma}.  We note that expressing the aforementioned volume as a simple formula is unfortunately not possible in full generality.  Of course, in any fixed case the volume of $\Omega_{L_{\bgamma}(\ba)}(H^{\bb})$ can be calculated explicitly, as has been done in the small cases considered in \cite{persal,dlB98,HB}.  The best we can achieve in full generality is to relate $\vol\Omega_{L_{\bgamma}(\ba)}(H^{\bb})$ to the simpler expression $V_{\bgamma}(\ba,\bb)$ in Theorem~\ref{thm:mainresult}, which we shall do presently.  So, with $\br_\ell$ as before, write
    \begin{equation*}
        \cT_{\ba,\bb}=\Big\{(t_{\br})_{\br}\in\R_+^{\cR_{\bgamma}(\ba)}:\sum_{\br\in\cR_{\bgamma}(\ba)}r_{i,j}t_{\br}\leqslant b_{i,j}\:\forall(i,j)\in\cI^{\boldm}\Big\},
    \end{equation*}
    and note that, by the change of variables $y_\ell=\exp(t_{\br_\ell}\log H)$, we have
    \begin{align}\label{eqn:omega}
        \vol\Omega_{L_{\bgamma}(\ba)}(H^{\bb})=(\log H)^{\#\cR_{\bgamma}(\ba)}\rint_{\cT_{\ba,\bb}}\prod_{\br\in\cR_{\bgamma}(\ba)}\exp(t_{\br}\log H)\:\dif t_{\br}.
    \end{align}
    Now, suppose that $\cX\subset\R_+^{\boldm}$ is a minimal basis for $\cR_{\bgamma}(\ba)$, and for each $\br\in\cR_{\bgamma}(\ba)$, define the unique real numbers $\theta_{\bx,\br}$ to be such that
    \begin{equation*}
        \sum_{\bx\in\cX}\theta_{\bx,\br}\bx=\br.
    \end{equation*}
    Without loss of generality, suppose that $\cX$ is chosen so that $\lambda_{\bx}=\theta_{\bx,\bb}>0$ for all $\bx\in\cX$, and such that
    \begin{equation*}
        \sum_{\bx\in\cX}x_{i,j}\sigma_{\bx}\leqslant b_{i,j}\:\forall(i,j)\in\cI^{\boldm}\quad\text{if and only if}\quad\sigma_{\bx}\leqslant\lambda_{\bx}\:\forall\bx\in\cX,
    \end{equation*}
    which clearly can be done.  Let $\cS_{\ba,\bb}(H)$ denote the set
    \begin{align*}
        \Big\{(\sigma_{\bx})_{\bx}\in\R_+^{\cX}:\sigma_{\bx}\leqslant\lambda_{\bx}\:\forall\bx\in\cX,\sum_{(i,j)\in\cI^{\boldm}}a_{i,j}\sum_{\bx\in\cX}x_{i,j}\sigma_{\bx}\geqslant\langle\ba,\bb\rangle-\frac{1}{\sqrt{\log H}}\Big\},
    \end{align*}
    and, for any $\bsigma=(\sigma_{\bx})_{\bx}\in\R_+^{\cX}$, define $A(\bsigma)=\vol\cA(\bsigma)$, where
    \begin{equation*}
        \cA(\bsigma)=\Big\{(t_{\br})_{\br}\in\R_+^{\cR_{\bgamma}(\ba)}:\sum_{\br\in\cR_{\bgamma}(\ba)}\theta_{\bx,\br}t_{\br}=\sigma_{\bx}\:\forall\bx\in\cX\Big\}.
    \end{equation*}
    Then, we may rewrite the integral on the right-hand side of \eqref{eqn:omega} as
    \begin{align}\label{e:hold}
        &\rint_{\cT_{\ba,\bb}}\prod_{\br\in\cR_{\bgamma}(\ba)}\exp(t_{\br}\log H)\:\dif t_{\br}\notag\\
        &\hspace{5mm}=\rint_{\R_+^{\cX}}\exp\Big(\log H\sum_{(i,j)\in\cI^{\boldm}}a_{i,j}\sum_{\bx\in\cX}x_{i,j}\sigma_{\bx}\Big)\:A(\sigma_{\bx})_{\bx}\prod_{\bx\in\cX}\dif\sigma_{\bx}\notag\\
        &\hspace{5mm}=\rint_{\cS_{\ba,\bb}(H)}\exp\Big(\log H\sum_{(i,j)\in\cI^{\boldm}}a_{i,j}\sum_{\bx\in\cX}x_{i,j}\sigma_{\bx}\Big)\:A(\sigma_{\bx})_{\bx}\prod_{\bx\in\cX}\dif\sigma_{\bx}\notag\\
        &\hspace{95mm}+O_{\ba,\bb}(H^{\langle\ba,\bb\rangle}e^{-\sqrt{\log H}})
    \end{align}
    Note that $\cA(\bsigma)$ is of dimension $\kappa_{\bgamma}(\ba)$ whenever $\sigma_{\bx}>0$ for all $\bx\in\cX$, by virtue of the fact that $\rank\cX=\#\cX=\rank\cR_{\bgamma}(\ba)$.  It is also clear that there exists a fixed embedding $\R_+^{\kappa_{\bgamma}(\ba)}\cong E\subset\R_+^{\cR_{\bgamma}(\ba)}$ such that $\cA(\bsigma)\subset E$ as $\bsigma\uparrow\blambda=(\lambda_{\bx})_{\bx}$, and thus$$\limsup_{\bsigma\uparrow\blambda}|A(\bsigma)-V_{\bgamma}(\ba,\bb)|=\limsup_{\bsigma\uparrow\blambda}|A(\bsigma)-A(\blambda)|=0,$$
    on account of the volumes both being measured in $E$ (cf.\:\S\,3.3 of \cite{AfCo}).  Notice that, by construction of the set $\mathscr{S}_{\ba,\bb}(H)$, we necessarily have $\mathscr{S}_{\ba,\bb}(H)\to\{\blambda\}$ as $H\to\infty$, from which the above implies that \begin{equation}\label{e:sup}\sup_{\bsigma\in\mathscr{S}_{\ba,\bb}(H)}|A(\bsigma)-V_{\bgamma}(\ba,\bb)|=o_{\ba,\bb}(1)\quad\text{as}\quad H\to\infty.\end{equation}By positivity of the integrand, we deduce from \eqref{e:hold} and \eqref{e:sup} that    
    \begin{align}\label{eqn:holder}
        &\rint_{\cT_{\ba,\bb}}\prod_{\br\in\cR_{\bgamma}(\ba)}\exp(t_{\br}\log H)\:\dif t_{\br}\notag\\
        &\hspace{5mm}=(V_{\bgamma}(\ba,\bb)+o_{\ba,\bb}(1))\rint_{\cS_{\ba,\bb}(H)}\prod_{\bx\in\cX}\exp\Big(\log H\sum_{(i,j)\in\cI^{\boldm}}a_{i,j}x_{i,j}\sigma_{\bx}\Big)\:\dif\sigma_{\bx}\notag\\
        &\hspace{95mm}+O_{\ba,\bb}(H^{\langle\ba,\bb\rangle}e^{-\sqrt{\log H}}).
    \end{align}
    Now, in order that it may become separable, we extend the integral on the right-hand side of \eqref{eqn:holder}, with negligible error, to derive
    \begin{align*}
        &\rint_{\cS_{\ba,\bb}(H)}\prod_{\bx\in\cX}\exp\Big(\log H\sum_{(i,j)\in\cI^{\boldm}}a_{i,j}x_{i,j}\sigma_{\bx}\Big)\:\dif\sigma_{\bx}\\
        &\hspace{20mm}=\prod_{\bx\in\cX}\rint_{-\infty}^{\lambda_{\bx}}\exp\Big(\log H\sum_{(i,j)\in\cI^{\boldm}}a_{i,j}x_{i,j}\sigma_{\bx}\Big)\:\dif\sigma_{\bx}+O_{\ba,\bb}(H^{\langle\ba,\bb\rangle}e^{-\sqrt{\log H}})\\
        &\hspace{20mm}=\frac{1}{(\log H)^{\#\cX}}\prod_{\bx\in\cX}\exp\Big(\log H\sum_{(i,j)\in\cI^{\boldm}}a_{i,j}x_{i,j}\lambda_{\bx}\Big)+O_{\ba,\bb}(H^{\langle\ba,\bb\rangle}e^{-\sqrt{\log H}})\\
        &\hspace{20mm}=\frac{H^{\langle\ba,\bb\rangle}}{(\log H)^{\#\cX}}+O_{\ba,\bb}(H^{\langle\ba,\bb\rangle}e^{-\sqrt{\log H}}),
    \end{align*}
    from which the assertion follows on appealing to \eqref{eqn:omega} and \eqref{eqn:holder}.
\end{proof} 
In order to use Theorems~\ref{thm:mainresult}~and~\ref{thm:mostgen} to prove the asymptotic formul\ae\:given in Theorems~\ref{thm:mainthm2}~and~\ref{thm:mainthm3}, we need only find appropriate choices for the tuples $\ba$, $\bb$, and $\bgamma$.  As Theorem~\ref{thm:mainthm2} can be viewed as a corollary of Theorem~\ref{thm:mainthm3}, we prove both simultaneously, proceeding as follows.
\begin{proof}[Proof of Theorems~\ref{thm:mainthm2}~and~\ref{thm:mainthm3}.] Firstly note that, with $\boldm=k\bone_m\oplus(1)$, we have\begin{align*}
   \energy_{\bgamma}(H^{\bb})=M_{m,k,\ell}(H^{\bc})\quad\text{where}\quad\bgamma= \bigoplus_{i\leqslant k} \bone_m \oplus(\ell)\quad\text{and}\quad  \bb = \bigoplus_{i\leqslant k} \bc\oplus(\|\bc\|/\ell),
\end{align*}
in which case
\begin{align*}
    \cR_{\bgamma} = \{ \br\in \Z_{+}^{\boldm} \colon \| \br_1 \| =\dots =\| \br_k \|=\ell\|\br_{k+1}\|\}.
\end{align*}
Now, our above choice of $\bb$ trivially satisfies the condition~\eqref{eqn:gammaisb}, and therefore to use Theorem~\ref{thm:mostgen}, we need only check that the condition~\eqref{eqn:gammaisa} is satisfied by the tuple
\begin{align*}
    \ba=\bigoplus_{i\leqslant k} \dfrac{1}{(k+1)\ell}\bone_m\oplus\Big(\frac{1}{k+1}\Big).
\end{align*}In order to prove that this condition is met, firstly note that, for all non-zero $\br \in \cR_{\bgamma}$, we must have $\| \br_i\|\geqslant\ell$ for all $i\leqslant k$, as well as $\|\br_{k+1}\|\geqslant1$, and thus\begin{align}\label{eqn:ar1}
    \langle \ba, \br \rangle = \frac{1}{(k+1)\ell}(\|\br_1\|+\cdots+\|\br_k\|)+\frac{1}{k+1}\|\br_{k+1}\|\geqslant1.
\end{align}
Moreover, it is clear that equality holds in \eqref{eqn:ar1} precisely when $\|\br_i\|=\ell$ for all $i\leqslant k$, in which case $\|\br_{k+1}\|=1$, and consequently
\begin{align*}
    \cR_{\bgamma}(\ba) = \Big\{  \Big(\sum_{t\leqslant\ell}\bu_t\Big)\oplus(1) \colon \bu_t\in\cU\:\forall t\leqslant \ell\Big\},
\end{align*} where
\begin{equation*}
    \cU\colonequals\Big\{\sum_{i\leqslant k}\e_{i,w_i}:w_i\in\{1,\ldots,m\}\:\forall i\leqslant k\Big\},
\end{equation*}
with $\e_{i,j}\in\cc^{\boldm}$ denoting the standard basis vectors.  Hence, it is clear that the condition \eqref{eqn:gammaisa} is satisfied by our choice of $\ba$.  Thus, appealing to Theorem~\ref{thm:mostgen}, we now need only show that
\begin{align}\label{eqn:cardcR1}
    \# \cR_{\bgamma}(\ba) = \binom{\ell+m-1}{m-1}^{\!\!k}\quad\text{and}\quad\rank\cR_{\bgamma}(\ba)=(m-1)k+1,
\end{align}and that, for any natural number $n$ and prime $p$, we have
\begin{equation}\label{eqn:nugammaaen}
    \nu_{\bgamma,\N^{\boldm}}(\ba;n,p)=\binom{\ell n+m-1}{m-1}^{\!\!k}.
\end{equation}
The first part of \eqref{eqn:cardcR1} is trivial, and for any natural number $n$ and prime $p$, we have
\begin{equation*}
    \nu_{\bgamma,\N^{\boldm}}(\ba;n,p)=\#\{\br\in\cU:\|\br_i\|=\ell n\:\forall i\leqslant k\},
\end{equation*}
from which \eqref{eqn:nugammaaen} follows using a standard combinatorial argument.  The second part of \eqref{eqn:cardcR1} will follow from the fact that $\rank\cR_{\bgamma}(\ba)=\rank\cU$ if we can show that
\begin{equation*}
	\cB\colonequals\bigcup_{i\leqslant k}\Big\{\e_{1,1}+\cdots+\e_{i-1,1}+\e_{i,w_i}+\e_{i+1,1}+\cdots+\e_{k,1}:w_i\in\{1,\ldots,m\}\Big\}
\end{equation*}
is linearly independent, and is moreover a basis for the span of $\cU$.  We shall return to the proof of these two statements later, but for now note that \eqref{eqn:cardcR1} implies that
\begin{align}\label{eqn:kappa1}
    \kappa_{\bgamma}(\ba) = \binom{\ell+m-1}{m-1}^{\!\!k}-(m-1)k-1,
\end{align}
where $\kappa_{\bgamma}(\ba)$ is as in Theorem~\ref{thm:mainresult}.  Thus, appealing to \eqref{eqn:cardcR1}, \eqref{eqn:nugammaaen}, and \eqref{eqn:kappa1} as well as the fact that
\begin{align*}
    \langle \ba,\bb\rangle = \frac{1}{(k+1)\ell}(\|\bc\|+\cdots+\|\bc\|)+\frac{1}{k+1}(\|\bc\|/\ell)= \|\bc\|/\ell,
\end{align*}
we see by Theorems~\ref{thm:mainresult}~and~\ref{thm:mostgen} that there exists a $\vartheta_{m,k,\ell,\bc}>0$ such that
\begin{align*}
    \energy_{\bgamma}(H^{\bb}) = H^{\|\bc\|/\ell} Q_{m,k,\ell,\bc}(\log H) + o(H^{\|\bc\|/\ell-\vartheta_{m,k,\ell,\bc}}),
\end{align*}
where $Q_{m,k,\ell,\bc}$ is a real polynomial satisfying
\begin{align*}
 &Q_{m,k,\ell,\bc}(\log H)\sim(\log H)^{\binom{\ell+m-1}{m-1}^{\!\!k}-(m-1)k-1}V_{\bgamma}(\ba,\bb)\\&\hspace{63mm}\times\prod_p(1-p^{-1})^{\binom{\ell+m-1}{m-1}^{\!\!k}}\sum_{n\geqslant0}\frac{1}{p^{n}}\binom{\ell n+m-1}{m-1}^{\!\!k}
\end{align*}
as $H\to\infty$.  The assertions then both follow on noting that $V_{\bgamma}(\ba,\bb)=V_{m,k,\ell,\bc}$ by our choice of parameters.  It thus remains only to prove that indeed the two desired properties of $\mathscr{B}$ hold.

(i) \textit{Linear independence of $\mathscr{B}$.}  To show that $\cB$ is a linearly independent subset of $\mathscr{U}$, it suffices to show that, if $\psi_{\bu}$ are real numbers satisfying
\begin{equation}\label{eqn:linindepproof}
	\sum_{\bu\in\cB}\psi_{\bu}\bu=\bigoplus_{i\leqslant k}\bzero_m,
\end{equation}
then we must have $\psi_{\bu}=0$ for all $\bu\in\cB$.  To prove that this is the case, we firstly note that if $\bu=\e_{1,w_1}+\cdots+\e_{k,w_k}\in\cB$ is such that $w_i\neq1$ for some $i\leqslant m$, then by definition of $\mathscr{B}$ we must have $u_{i,w_i}=1$ and $v_{i,w_i}=0$ for all $\bv\in\cB\backslash\{\bu\}$, forcing $\psi_{\bu}=0$.  So, we must have $\psi_{\bu}=0$ in \eqref{eqn:linindepproof} for all $\bu\in\cB\backslash\{\e_{1,1}+\cdots+\e_{k,1}\}$, which is sufficient to prove that $\cB$ is linearly independent.

(ii) \textit{The set $\mathscr{B}$ is a basis for the span of $\mathscr{U}$.}  It is trivial that $\mathscr{B}\subset\mathscr{U}$, so it suffices to prove that $\vecspan\cU\subseteq\vecspan\mathscr{B}$.  Firstly, consider the subsets
\begin{equation*}
    \cU(n)\colonequals\Big\{\sum_{i\leqslant k}\e_{i,w_i}\in\cU:\#\{i\leqslant k:w_i\neq1\}=n\Big\},
\end{equation*}
which clearly satisfy
\begin{equation*}
    \bigcup_{n\leqslant k}\cU(n)=\cU.
\end{equation*}
Explicitly, $\cU(n)$ is the set of tuples $\e_{1,w_1}+\cdots+\e_{k,w_k}\in\cU$ such that $w_i\neq1$ for exactly $n$ distinct columns $i$.  We shall now prove that $\cU(n)\subseteq\vecspan\cB$ for all $n\leqslant k$ by induction on $n$.  Firstly, we note that $\cU(n)\subseteq\cB$ for the base cases $u=0,\:1$.  Next, suppose that $n\geqslant2$ and that $\cU(n')\subseteq\vecspan\cB$ for all $n'<n$.  Then, for any $\bu\in\cU(n)$ there exist distinct $i_1,i_2\leqslant k$ with $w_{i_1},w_{i_2}\neq1$ such that
\begin{align*}
    \bu =\e_{i_1,w_{i_1}}+\e_{i_2,w_{i_2}}+\bq\quad \text{where} \quad\bq=\sum_{\substack{i\leqslant k\\i\neq i_1,i_2}}\e_{i,w_i}
\end{align*}
with $w_i\neq1$ for exactly $(n-2)$-many columns $i\in\{1,\ldots,k\}\backslash\{i_1,i_2\}$.  Notice that\begin{align*}
&\bu = (\e_{i_1,1}+\e_{i_2,w_{i_2}}+\bq)+(\e_{i_1,w_{i_1}}+\e_{i_2,1}+\bq) - (\e_{i_1,1}+\e_{i_2,1}+\bq), 
\end{align*}
where the first and second bracketed terms are elements of $\cU(n-1)$, while the last term is an element of $\cU(u-2)$.  Consequently, we have
\begin{equation*}
\cU(n)\subseteq\vecspan\big(\cU(n-1),\cU(n-2)\big)\subseteq\vecspan\cB,
\end{equation*}
the last inclusion following from the induction hypothesis.  

Having established that $\mathscr{B}$ is linearly independent, and moreover a basis for the span of $\mathscr{U}$, we have completed the proof of Theorems~\ref{thm:mainthm2} and \ref{thm:mainthm3}.
\end{proof}
Note that the appropriate choice of $\ba$ for use in Theorem~\ref{thm:mostgen} is not always unique, though there are of course choices that are more natural than others.\newline

\section{Some concluding remarks.}\label{sec:rem}
In this final section, we outline some further applications of Theorems~\ref{thm:mainthm2}~and~\ref{thm:mainthm3}, as well as briefly discuss some related problems.\newline
\subsection{Some non-explicit results.}\label{sec:nonexplicit}
For general choices of $\bgamma$, it may not always be possible to determine explicitly the degree and leading coefficient of the polynomial $Q$ in Theorem~\ref{thm:mainresult}, on account of the choice of $\ba$ not satisfying the conditions of Theorem~\ref{thm:mostgen}.  Nonetheless, in some of these cases, it still may be possible to obtain an asymptotic for $\energy_{\bgamma,\cN}(H^{\bb})$ in terms of a polynomial $Q$ of non-explicit degree and leading coefficient.  Indeed, we have the following result, which improves Theorem~\ref{thm:mainresult} under a minor additional condition.
\begin{prop}\label{prop:51}
    Suppose that $\lambda>0$ is such that, for all $\eta>0$ there exists a constant $c_{\eta}>0$ such that\begin{align}\label{eqn:cHlambda}
        \energy_{\bgamma,\cN}(H^{\bb})\geqslant c_\eta H^{\lambda-\eta}
    \end{align} for infinitely-many $H$.  If there exists an $\ba$ satisfying the hypothesis of Theorem~\ref{thm:mainresult}, such that $\langle\ba,\bb\rangle=\lambda$, then there exists a polynomial $R$, not identically zero, and a constant $\vartheta>0$ such that
    \begin{equation*}
        \energy_{\bgamma,\cN}(H^{\bb})=H^\lambda R(\log H)+o(H^{\lambda-\vartheta})\quad\text{as}\quad H\to\infty.
    \end{equation*}
\end{prop}
The main advantage of the above result is that the polynomial $R$ is necessarily non-zero for all sufficiently large $H$, as opposed to the polynomial $Q$ in Theorem~\ref{thm:mainresult} which may be identically zero.  Thus, the additional condition \eqref{eqn:cHlambda} allows us to establish an asymptotic formula, rather than simply a majorisation.  Of course, an upper bound for the degree of $R$ can easily be determined if $\ba$ is known explicitly.

As with Theorems~\ref{thm:mainthm2}~and~\ref{thm:mainthm3}, we may also apply Proposition~\ref{prop:51} to study certain types of restricted divisor functions. For example, consider a generalisation of the divisor finction $\tau_m(n;\mathbf{X})$, given by\begin{equation*}
    \tau_{\bfxi}(n;\mathbf{X})=\#\{\bd\in\fB^{m}(\mathbf{X}):d_1^{\xi_1}\cdots d_\ell^{\xi_m}=n\},
\end{equation*}
where $\bfxi\in\N^m$ for some fixed natural number $m$.  It is not difficult to see, similarly to as in the proof of Theorems~\ref{thm:mainthm2}~and~\ref{thm:mainthm3}, that we have the relation
\begin{equation*}
    \energy_{\bgamma}(H^{\bb})=\sum_{n}\tau_{\bfxi}(n;H^{\bc})^k\quad\text{where}\quad\bgamma=\bigoplus_{i\leqslant k}\bfxi\quad\text{and}\quad\bb=\bigoplus_{i\leqslant k}\bc,
\end{equation*}
with $\boldm=k\bone_m$.  We additionally note that 
\begin{equation*}
    \sum_{n}\tau_{\bfxi}(n;H^{\bc})=\#\fB^m(H^{\bc})\sim H^{\|\bc\|}
\end{equation*}
as $H\to\infty$, and thus
\begin{equation}\label{eqn:loweroh}
    \sum_{n}\tau_{\bfxi}(n;H^{\bc})^k\geqslant(1+o(1))H^{\|\bc\|}
\end{equation}
for any natural number $k\geqslant2$.  Now, following the same argument as in the proof of Theorem~\ref{thm:mainthm2}, we see that $\ba=(1/k)\bone_{\boldm}$ satisfies the hypothesis of Theorem~\ref{thm:mainresult} for our choice of $\bgamma$ and $\bb$.  Consequently, appealing to \eqref{eqn:loweroh} and the fact that $\langle\ba,\bb\rangle=\|\bc\|$, we see by Proposition~\ref{prop:51} that
\begin{equation}\label{eqn:thm261}
    \sum_n\tau_{\bfxi}(n;H^{\bc})^{k}\sim H^{\|\bc\|}R_{\bfxi,k,\bc}(\log H)\quad\text{as}\quad H\to\infty
\end{equation}
for some polynomial $R_{\bfxi,k,\bc}$, not identically zero.  Slightly adapting the argument used to prove Theorems~\ref{thm:mainthm2}~and~\ref{thm:mainthm3}, we can show that this polynomial moreover satisfies
\begin{align*}
    &\deg R_{\bfxi,k,\bc}\leqslant\sum_n\#\{j\leqslant m:\xi_j=n\}^k\\&\hspace{35mm}-\#\{\xi_j:j\leqslant m\}-\sum_nk\max(0,\#\{j\leqslant m:\xi_j=n\}-1),
\end{align*}
from which we derive the following.
\begin{prop}
    If $\bfxi\in\N^m$ has distinct ordinates, then there exists a constant $D_{\bfxi,k,\bc}>0$ and a real number $\vartheta_{\bfxi,k,\bc}>0$ such that
    \begin{equation*}
        \sum_n\tau_{\bfxi}(n;H^{\bc})^k=D_{\bfxi,k,\bc} H^{\|\bc\|}+o(H^{\|\bc\|-\vartheta_{\bfxi,k,\bc}})\quad\text{as}\quad H\to\infty.
    \end{equation*}
\end{prop}
The above is an example of where the lower bound \eqref{eqn:lowerbound} is precisely of the correct order.  Note that applying Theorem~\ref{thm:mostgen}  to improve~\eqref{eqn:thm261} is not directly possible in this case, due to unavailability of a tuple $\ba$ satisfying the condition \eqref{eqn:gammaisa} of Theorem~\ref{thm:mostgen}.  The condition \eqref{eqn:gammaisa} itself is necessary in order for the condition \eqref{eqn:gstarcond} of Lemma~\ref{lem:dlB2} to be fulfilled.  This issue appears for most choices of $\bgamma$ that are not of a sufficiently nice form.  In addition, for many choices of $\bgamma$, we have $\min(\|\bb_1\|,\ldots,\|\bb_k\|)>\mu_{\bgamma,\bb}$, in which case, the trivial bounds
\eqref{eqn:upperbound} and \eqref{eqn:lowerbound} differ by a factor greater than $O_\varepsilon(H^\varepsilon)$.\newline

\subsection{Imposing further restrictions.}\label{sec:coprime}
As a porism of Theorem~\ref{thm:mostgen}, we have
\begin{equation}\label{eqn:ratioeqn}
    \energy_{\bgamma,\cN}(H^{\bb})\sim A_{\bgamma,\cN}\energy_{\bgamma}(H^{\bb})\quad\text{as}\quad H\to\infty,
\end{equation}
where
\begin{equation*}
    A_{\bgamma,\cN}=\prod_p\Big(\sum_{n\geqslant0}\frac{\nu_{\bgamma,\cN}(\ba;n,p)}{p^n}\Big)\Big(\sum_{n\geqslant0}\frac{\nu_{\bgamma}(\ba;n)}{p^n}\Big)^{-1},
\end{equation*}
with $\nu_{\bgamma}(\ba;n)=\#\{\br\in\cR_{\bgamma}:\langle\ba,\br\rangle=n\}$.  So e.g., fixing a $q\in\N$, and choosing $\cN$ to be the set of natural numbers prime to $q$, we may use \eqref{eqn:MA} and \eqref{eqn:ratioeqn} to show that
\begin{equation*}
    \sum_{(n,q)=1}\tau_2(n;(X,X))^2=\frac{2}{\zeta(2)}\prod_{p|q}\frac{(p-1)^3}{p^2(p+1)}X^2\log X+O(X^2)\quad\text{as}\quad X\to\infty.
\end{equation*}
More generally, we can use \eqref{eqn:ratioeqn} along with \eqref{eqn:gchigamma} to count the number of integer solutions to the multiplicative Diophantine equation \eqref{eqn:x1xk} where only select ordinates are restricted.  We simply offer the above example as a way of illustrating to the reader how this procedure can easily be followed by appealing to Theorem~\ref{thm:mostgen}.

We note that our method does not, unfortunately, allow us to count solutions to \eqref{eqn:x1xk} over the Piatetski-Shapiro or Beatty numbers, as in these cases $\cN$ is not closed under the Hadamard product.  This is also the main obstacle when trying to count solutions to \eqref{eqn:x1xk} lying in an incomplete set of residue classes modulo $\mathbf{q}$, for some fixed integer $\boldm$-tuple $\mathbf{q}$.  Of course, Theorem~\ref{thm:mostgen} can be applied to the nicer of these cases, such as when $\cN$ is taken to be the set of quadratic residues modulo $\bq$, or the set $\bone_{\boldm}+\bq\N$.  Additionally, if $\cN$ is too sparse, then the product \eqref{eqn:notfriable} will not converge for any $\sigma\in(0,1)$, in which case Theorem~\ref{thm:mainresult} is not applicable.  Such is the case with friable integers and the Piatetski-Shapiro sequence.\newline

\subsection{Applications to counting singular matrices.}\label{sec:singularmat}Using Theorem~\ref{thm:mainthm2}, we may show that there exists a constant $\vartheta>0$ such that
\begin{equation}\label{eqn:manew}
    M_{2,2}(X,X)=\frac{2}{\zeta(2)}X^2\log X+AX^2+o(X^{2-\vartheta})\quad\text{as}\quad X\to\infty
\end{equation}
for some constant $A$.  Of course, the constant $A$ is the same as the one given in \eqref{eqn:MA}, which supersedes \eqref{eqn:manew} in that the power-saving is explicit.  Indeed, for the smaller cases such as this, elementary methods tend to give stronger results than the more general machinery of La~Bret\`eche, which is best-suited to the larger systems of variables considered in this article.

We now note that \eqref{eqn:manew}, and thus \eqref{eqn:MA}, improves on a preliminary result of the first-named author in the recent work \cite{Af} on the arithmetic statistics of the set $\cD_2(H,\Delta)$ of $2\times 2$ integer matrices of determinant $\Delta$ whose entries are integers of absolute value not exceeding $H$.  Asymptotic formul\ae\:for $\#\cD_2(H,\Delta)$ when $\Delta\neq0$ have been given in~\cite{Af,GG1,GG2,Guria} for a large range of $\Delta$ relative to $H$.  For the case $\Delta=0$, the first-named author~\cite{Af} showed that the number $\#\mathscr{D}_2(H,0)$ is related to the moment $M_{2,2}(H,H)$ by virtue of the asymptotic relation\begin{align*}
    \#\cD_2(H,0)=8M_{2,2}(H,H)+16H^2+O(H),
\end{align*} 
from which we derive, appealing to \eqref{eqn:MA}, the following.
\begin{prop}\label{cor:det}As $H\to \infty$, we have
    \begin{align*}
    \#\cD_2(H,0)=\dfrac{16}{\zeta(2)} H^2\log H + (8A+16)H^2+O(H^{19/13}(\log H)^{7/13}),
\end{align*} where $A$ is as in \eqref{eqn:MA}.
\end{prop}
A related problem is that of counting the number of singular $2\times 2$ matrices whose entries are Egyptian fractions of a bounded height.  Precisely, this problem is concerned with the arithmetic statistics of the set \begin{align*}
    \cD_2(\Z^{-1};H,0)=\Big\{(a,b,c,d)\in\Z^4\cap\big([-H,H]\backslash\{0\}\big)^4:\det\begin{pmatrix}
        1/a &1/b\\1/c&1/d
    \end{pmatrix}=0   \Big\}.
\end{align*}
Now, it was shown in~\cite{AKOS} that \begin{align*}
    H^2 \ll \#\cD_2(\Z^{-1};H,0) \lle H^{2+\varepsilon},
\end{align*}
which we can sharpen, by appealing to Proposition~\ref{cor:det} and the fact that
\begin{equation*}
    \sum_{\substack{ad=0=bc\\|a|,|b|,|c|,|d|\leqslant H}}1=16H^2+O(H),
\end{equation*}
to the following asymptotic formula.
\begin{cor}\label{cor:detdet}As $H\to \infty$, we have
    \begin{align*}
    \#\cD_2(\Z^{-1};H,0) =\dfrac{16}{\zeta(2)} H^2\log H + 8AH^2+O(H^{19/13}(\log H)^{7/13}),
\end{align*} where $A$ is as in \eqref{eqn:MA}.    
\end{cor}
Note that, Proposition~\ref{cor:det} and Corollary~\ref{cor:detdet} should be viewed as consequences of the estimate \eqref{eqn:MA}, rather than as applications of the main results of this article.\newline

\subsection{Generalisations to other $L$-functions.}\label{sec:lfunct}
It is well-known that the number $\tau_m(n)$ is precisely the $n^{th}$ Dirichlet coefficient of $\zeta(s)^m$, and we see that $\tau_m(n;\mathbf{X})$ is analogously related to a product of $m$-many truncations of the Dirichlet series for the $\zeta$-function.  Indeed, we clearly have the relation$$\Big(\sum_{d\leqslant X_1}d^{-s}\Big)\cdots\Big(\sum_{d\leqslant X_m}d^{-s}\Big)=\sum_{n}\tau_m(n;\mathbf{X})n^{-s}.$$ We can more generally consider moments of the coefficients of Dirichlet polynomials which arise from products of truncations of other $L$-functions.  For example, consider a collection of arithmetical objects $f_1,\ldots,f_m$, each to which an $L$-function $L(s,f_j)$ can be associated.  The forthcoming work \cite{Cnew} of the second-named author is partly concerned with the moments of weighted restricted divisor functions of the type $$\mathop{\sum_{d_1\leqslant X_1}\cdots\sum_{d_m\leqslant X_m}}_{d_1\cdots d_m=n}\frac{\lambda_{f_1}(d_1)\cdots \lambda_{f_m}(d_m)}{d_1^{\alpha_1}\cdots d_m^{\alpha_m}},$$ where $\lambda_{f_j}(d_j)$ denotes the $d_j^{th}$ Dirichlet coefficient of $L(s,f_j)$, and the $\alpha_j$ are fixed real shifts.  Of course, one can only hope to get an asymptotic formula in the case where the map $(s_1,\ldots,s_m)\mapsto L(s_1+\alpha_1,f_1)\cdots L(s_m+\alpha_m,f_m)$ has a pole of sufficiently large order at some point $(s_1,\ldots,s_m)\in\cc^m$.\newline
 
\subsection{Other future directions.}
The proof of Theorem~\ref{thm:mostgen} can be adapted to instead count the number of solutions $\bx\in\fB^{\boldm}(H^{\bb})\cap\cN$ to equations of the form \begin{align}\label{eqn:alphax}
    \alpha_1\bx_1^{\bgamma_1}\equiv\cdots\equiv\alpha_k\bx_k^{\bgamma_k}\pmod{q}
\end{align}for some fixed $\balpha\in\N^k$ and $q\in\N$.  This problem was previously considered by Ayyad, Cochrane and Zheng~\cite{Ayyad}, Garaev and Garcia~\cite{GarGar}, and Kerr~\cite{Kerr} for the congruence $x_1x_2\equiv x_3x_4$ modulo a prime $p$.  Now, while the function $f$ corresponding to \eqref{eqn:alphax} for use in Lemmata~\ref{lem:dlB} and~\ref{lem:dlB2} is clearly not multiplicative, it is possible to split the problem into a finite number of counting problems of the form
\begin{equation*}
    \energy_{q,\bgamma,\cN}(\mathbf{U})\colonequals\#\big\{\by\in\fB^{\boldm}(\mathbf{U})\cap\cN:\by_1^{\bgamma_1}\equiv\cdots\equiv\by_k^{\bgamma_k}\!\!\!\!\pmod{q}\big\},
\end{equation*}
where $U_{i,j}=H^{b_{i,j}}/h_{i,j}$ for some small solution $\mathbf{h}$ to the congruence \eqref{eqn:alphax}.  Even more generally, we could consider the problem of counting the number of integer solutions to~\eqref{eqn:alphax} lying in the parallelepiped
\begin{equation*}
    \fP_{\cN}(\mathbf{X},\mathbf{Y})\colonequals\big\{\bd\in\cN: X_{i,j}<d_{i,j}\leqslant X_{i,j}+Y_{i,j}\:\forall(i,j)\in\cI^{\boldm}\big\},
\end{equation*}
where $\mathbf{X},\mathbf{Y}\in\R^{\boldm}$.  In particular, it would be interesting to determine the asymptotic behaviour of this number as each $X_{i,j}$ and $Y_{i,j}$ independently tend to infinity.  

Another related problem, is the study of the distribution of integer solutions to  \begin{align}\label{eqn:xmym}
    x_1\dots x_m-y_1\dots y_m=\Delta,
\end{align}where $\Delta \neq 0$.  The machinery of La~Bret\`eche cannot be directly applied to study the equation \eqref{eqn:xmym}, on account of the fact that the corresponding solution set is not closed under the Hadamard product.  Nonetheless, using different methods, the first-named author~\cite{Af} and Ganguly and Guria~\cite{GG1,GG2} have, in the case $m=2$, given asymptotic formul\ae\:for the number of solutions to~\eqref{eqn:xmym} in the box $[-H,H]^4$, as $H$ and $|\Delta|$ independently tend to infinity.

Lastly, motivated by the applications of such results in the works \cite{silverman,alina,alinac}, Pappalardi, Sha, Shparlinski, and Stewart \cite{pappalardi} studied the asymptotic behaviour of the number of multiplicatively-dependant $n$-tuples of algebraic integers of bounded height, with respect to different height functions.  Recall, for a fixed multiplicative group $G$, an $n$-tuple $(\nu_1,\ldots,\nu_n)\in G^n$ is said to be multiplicatively-dependant if there exists a $\bfxi\in\zz^n\backslash\{\bzero_n\}$ such that $\nu_1^{\xi_1}\cdots\nu_n^{\xi_n}=\text{id}_G$, where $\text{id}_G$ denotes the multiplicative identity in $G$.  In a similar direction, the number $$\#\{\bnu \in (\Z\cap [-H,H])^n\colon  \balpha\cdot \bnu =J, \bnu \:\text{is multiplicatively-dependant}\}$$ for a fixed $\balpha\in \Z^n$, $J\in \Z$ is the topic of study in the forthcoming work \cite{afifnew} of the first-named author with Iverson and Sanjaya.  This leads us to consider generalisations of our own work, such as the problem of determining the asymptotic behaviour of
\begin{equation*}
    \energy_{\cC,\cN}(H^{\bb})\colonequals\#\big\{\bx\in\fB^{\boldm}(H^{\bb})\cap\cN:\bx_1^{\bgamma_1}=\cdots=\bx_k^{\bgamma_k}\:\exists\bgamma\in\cC\big\}
\end{equation*}
as $H$ tends to infinity, where $\cC\subset\N^{\boldm}$ is fixed.  Now, the method of proof used in \cite{pappalardi} does not appeal to the work of La~Bret\`eche \cite{dlB,dlBcompter}, but rather makes critical use of a result of Widmer \cite{widmer}.  This is due to the fact that, in their setting, it is not possible to relate any sum of the form $S_f(X^{\bb})$ to the number of interest.  Nonetheless, it is likely that Lemmata~\ref{lem:dlB}~and~\ref{lem:dlB2} of the present article could be used to derive an asymptotic estimate for $\energy_{\cC,\cN}(H^{\bb})$ in the case where $\cC$ is sufficiently well-structured.\newline

\section*{Acknowledgements.} The authors would like to thank Alina Ostafe, Igor E. Shparlinski, and Liangyi Zhao for supports and suggestions during the preparation of this work.  We would also like to thank R\'{e}gis de~la~Bret\`eche, as well as Marc Munsch and Igor E. Shparlinski, for helpful correspondences regarding their articles \cite{dlB}, and \cite{MSh}, respectively.  Thanks are due to Marc Munsch for bringing the articles \cite{louboutin,Ayyad} to our attention, and to the anonymous referee for such a detailed report with many helpful suggestions.  

During the preparation of this work, the first-named author was supported by the Australian Research Council Grant DP230100530 and a University of New South Wales Tuition Fee Scholarship, while the second-named author was supported by the Australian Government, the University of New South Wales, and New College, through R.T.P.\:, U.P.A., and Ph.D.\:Scholarships, respectively.\newline
 
\renewcommand*{\mkbibnamefamily}[1]{\textsc{#1}}
\renewcommand*{\mkbibnameprefix}[1]{\textsc{#1}}

\centerline{\textit{References.}}
\emergencystretch 2em
\printbibliography[heading=none]

\end{document}